\documentclass{aims}
\usepackage{amsmath}
  \usepackage{paralist}
  \usepackage{graphics} %% add this and next lines if pictures should be in esp format
  \usepackage{epsfig} %For pictures: screened artwork should be set up with an 85 or 100 line screen
 \usepackage[colorlinks=true]{hyperref}
\hypersetup{urlcolor=blue, citecolor=red}

\def\currentvolume{X}
 \def\currentissue{X}
  \def\currentyear{200X}
   \def\currentmonth{XX}
    \def\ppages{X--XX}

\newtheorem{theorem}{Theorem}[section]

\newtheorem{example}{Example}

\newtheorem{lemma}[theorem]{Lemma}

\theoremstyle{definition}
\newtheorem{definition}[theorem]{Definition}
\newtheorem{remark}{Remark}

\newcommand{\inte }{{\rm int}\,}

\newcommand{\dom }{\,{\rm dom}\,}

\newcommand{\Id }{\mbox{Id} \,}

\title[Cone Conditions for Invariant Manifolds]
      {Cone Conditions and Covering Relations for
Topologically Normally Hyperbolic Invariant Manifolds}

\author[Maciej J. Capi\'nski and Piotr Zgliczy\'nski]{}


 \email{mcapinsk@agh.edu.pl}
 \email{zgliczyn@ii.uj.edu.pl}

\thanks{The research of the first author was partially supported
by the Polish Ministry of Science and Higher Education.}

\thanks{The first author would like to thank the African Institute For Mathematical Sciences, Muizenberg, Cape Town, South Africa, where the work was partially developed.}

\thanks{The second author is supported by the Polish State Ministry of
            Science and Information Technology  grant N201 024 31/2163}

\thanks{Both authors are supported by the Polish State Ministry of
            Science and Information Technology  grant N201 543238}
\begin{document}
\maketitle

% Enter the first author's name and address:
\centerline{\scshape Maciej J. Capi\'nski }
\medskip
{\footnotesize
% please put the address of the first author
 \centerline{AGH University of Science and Technology, Faculty of Applied Mathematics}
%   \centerline{Other lines}
   \centerline{ al. Mickiewicza 30, 30-059 Krak\'ow, Poland}
} % Do not forget to end the {\footnotesize by the sign }

\medskip

\centerline{\scshape Piotr Zgliczy\'nski}
\medskip
{\footnotesize
 % please put the address of the second  and third author
 \centerline{ Jagiellonian University, Institute of Computer Science,}
  % \centerline{Other lines}
   \centerline{\L ojasiewicza  6, 30--348  Krak\'ow, Poland}
}

\bigskip

% The name of the associate editor will be entered by an editorial staff
% \centerline{(Communicated by the associate editor name)}

%The abstract of your paper
\begin{abstract}
We present a topological proof of the existence of invariant manifolds for maps with normally hyperbolic-like properties. 
The proof is conducted
 in the phase space of the system. In our approach we do not require that the map is a perturbation of some other
 map for which we already have an invariant manifold. We provide conditions which imply the existence of
  the manifold within an investigated region of the phase space. The required assumptions are formulated in
  a way which allows for rigorous computer assisted verification. We apply our method to obtain
   an invariant manifold within an explicit range of parameters for the rotating H\'enon map.
\end{abstract}

%TCIDATA{OutputFilter=latex2.dll}
%TCIDATA{Version=5.00.0.2606}
%TCIDATA{LaTeXparent=0,0,online-edit.tex}
\section{Introduction}

The goal of our paper is to present a topological 
proof of the existence of invariant manifolds for maps with normally hyperbolic type properties,
in a vicinity of an approximate invariant
manifold.  In our opinion there are two main advantages of our
approach: 1) we do not assume that the given map is a
perturbation of some other map for which we have a normally
hyperbolic invariant manifold, 2) the assumptions could be
rigorously checked with computer assistance if our
approximation of the invariant manifold is good enough.

Our results about invariant manifolds  are weaker
than the standard results for normal hyperbolicity \cite{BLZ5,Ch,HPS,Wi}. In fact we do not prove
any smoothness result and the existence of fibration of the stable
and unstable manifolds of the invariant manifold. We believe though that topological assumptions 
considered by us should be sufficient to prove normal hyperbolicity in the usual sense.
This paper should be considered as a first step towards this end.
In a subsequent publication we intend to improve our results in this
direction.

In the standard approach to the proof of  various invariant
manifold theorems all considerations are done in suitable function
spaces or sequences spaces, moreover the existence of the
invariant manifold for nearby map (or ODE) is always assumed, see
for example \cite{Ch,HPS,Wi} and the references given there.
Usually these proofs do not give any computable bounds for the
size of perturbation for which the invariant manifold
exists, with only exception known to us being the result of Bates, Lu and Zeng \cite{BLZ5}. 

In contrast to the above mentioned standard approach, in our
method  the whole proof is made in the phase space. This method of proof of the existence of
invariant manifolds is not entirely new, see
for example the proof of Jones \cite{J}  in the context of
slow-fast system of ODEs. But still he considered the perturbation
of some normally hyperbolic invariant manifold. In \cite{Ca} a somewhat similar 
approach has been applied to obtain a topologically normally hyperbolic invariant set.
 The result relied only on covering relations without the use of cone conditions. There 
it has not been shown that the invariant set is a manifold, hence the result was weaker than the one presented in this paper.  

The work is organised as follows. In Section \ref{sec:prel} we introduce notations and give preliminaries on vector bundles.
 In Section \ref{sec:ch-sets} we introduce central-hyperbolic sets (ch-sets), covering relations and cones. 
A ch-set will play the role of a region in which we suspect to find an invariant manifold. Covering relations 
will ensure the existence of an invariant set within a ch-set. In Section \ref{sec:main-th} we introduce cone conditions for maps 
and main results. Cone conditions combined with covering relations will give the existence of a normally hyperbolic-like invariant manifold 
inside of a ch-set. In Section \ref{sec:ver-cond} we show how cone conditions and covering relations can be verified in practice.
 In Section \ref{sec:rel-classic} we discuss how our result relates so far to the classical normally hyperbolic invariant manifold theorem.
 To demonstrate clearly the strength of our approach, in
Section~\ref{sec:Henon} we prove that for the rotating H\'enon map
considered in \cite{HL} for an explicit range of parameters there
exists an invariant manifold.

%TCIDATA{OutputFilter=latex2.dll}
%TCIDATA{Version=5.00.0.2606}
%TCIDATA{LaTeXparent=0,0,online-edit.tex}
\section{Preliminaries} \label{sec:prel}
\subsection{Notation}

By $\Id:X \to X$ we will denote the identity map.

Let  $\mathbb{R}^n$ be equipped with some norm $\|\cdot \|$. For
$x_0 \in \mathbb{R}^n$ and $r>0$ we define $B_n(x_0,r)=\{ x \in
\mathbb{R}^n \ | \  \|x-x_0 \| < r  \}$. Since most of the time we
will be using unit balls centered at $0$, we set $B_n=B_n(0,1)$.
When considering a linear map on $A:\mathbb{R}^n
\mathbb{R}^k$ the symbol $\|A\|$ will always denote the standard
operator norm of $A$, i.e. $\|A\|=\sup_{\|x\|=1}\|Ax\|$. We will
also use $\|A\|_m=\inf_{\|x\|=1} \|Ax\|$.

Let $\Pi X=X_1\times X_2 \times \dots \times X_n$. For
$i=1,\dots,n$ we define projection $\pi_i:\Pi X \to X_i$ by
$\pi_i(x_1,\dots,x_n)=x_i$. Sometimes when the variables in the
cartesian product have names, for example $(x,y,z)$, we will use
names of variables as indices for projections, hence
$\pi_y(x,y,z)=y$.

Let $\mathcal{U}=\{U_i\}_{i  \in I}$ and $\mathcal{V}=\{V_j\}_{j
\in J}$ be coverings of a set $X$ (i.e. $X=\bigcup
\mathcal{U}=\bigcup \mathcal{V}$), we say that covering
$\mathcal{V}$ is inscribed in $\mathcal{U}$, denoted by
$\mathcal{V} \prec \mathcal{U}$, iff for any $j \in J$ there
exists $i \in I$, such that $V_j \subset U_i$. If $X$ is a
topological space, then we say that covering $\mathcal{U}$ is open
iff it consists from open sets.

Let $U$ be a topological space. We say that $U$ is contractible,
when there exists a deformation retraction of $U$ onto single
point space, i.e.  there exists continuous map $h:[0,1] \to U$ and
$x_0 \in U$, such that $h(0,q)=q$  and $h(1,q)=x_0$ for all for $q
\in U$.

For $x,y \in \mathbb{R}^n$ we set $[x,y]=\{ z=tx+(1-t)y, \ t \in
[0,1] \}$.

%TCIDATA{OutputFilter=latex2.dll}
%TCIDATA{Version=5.00.0.2606}
%TCIDATA{LaTeXparent=0,0,online-edit.tex}

\subsection{Vector bundles}

We start with recalling the definition of the vector bundle
\cite{Hi}.
\begin{definition}
Let $B,E$ be topological spaces. Let $p:E \to B$  be a continuous
map. A \emph{vector bundle chart } on $(p,E,B)$ with \emph{domain
} $U$ and \emph{dimension} $n$ is a homeomorphism
$\varphi:p^{-1}(U) \to U \times \mathbb{R}^n$, where $U \subset B$
is open and such that
\begin{equation}
  \pi_1 \circ \varphi (z)= p(z), \qquad \mbox{for $z \in
  p^{-1}(U)$}.
\end{equation}
We will denote such bundle chart by a pair  $(\varphi,U)$.

For each $x\in U$ we define the homeomorphism $\varphi_x$ to be
the composition
\begin{equation}
  \varphi_x:p^{-1}(x) \stackrel{\varphi}{\rightarrow} \{x\} \times
  \mathbb{R}^n \to \mathbb{R}^n .
\end{equation}
 A \emph{vector bundle atlas $\Phi$} on $(p,E,B)$ is a family of
vector bundle charts on $(p,E,B)$ with the values in the same
$\mathbb{R}^n$, whose domains cover $B$ and such that whenever
$(\varphi,U)$  and $(\psi,V)$ are in $\Phi$ and $x \in U \cap V$,
the homeomorphism $\psi_x \varphi_x^{-1}: \mathbb{R}^n \to
\mathbb{R}^n$ is linear. The map
\begin{equation}
  U \cap V \ni x \mapsto \psi_x \varphi_x^{-1} \in GL(n)
\end{equation}
is continuous for all pairs of charts in $\Phi$.

A maximal vector bundle atlas $\Phi$ is a \emph{vector bundle
structure} on $(p,E,B)$.

We then call $\xi=(p,E,B,\Phi)$ a \emph{vector bundle} having
\emph{(fibre) dimension $n$, projection $p$, total space $E$} and
\emph{base space $B$}. Often $\Phi$ will not be explicitly
mentioned. In fact we may denote $\xi$ by $E$. Sometimes it is
convenient to put $E=E\xi$, $B=B\xi$, etc.

The \emph{fibre} over $x \in B$ is the space
$p^{-1}(x)=\xi_x=E_x$. $\xi_x$ has the vector space structure.

If the $E,B$ are $C^r$ manifolds and all maps appearing in the
above definition are $C^r$, then we will say that the bundle
$(p,E,B,\Phi)$ is a $C^r$-bundle.

\end{definition}

One can  introduce the notion of subbundles, morphisms etc (see
\cite{Hi} and references given there). The fibers can have a
structure: for example a scalar product, a norm, which depend
continuously on the base point.

\begin{definition}
\label{def:Banachbundle} We say that the vector bundle $\xi$ is a
Banach vector bundle with fiber being the Banach space
$(F,\|\cdot\|)$, if for each $x \in B\xi$ the fiber $\xi_x$ is a
Banach space with norm $\| \cdot\|_x$ such that for each bundle
chart $(\varphi,U)$ the map $\varphi_x : E_x \to F $ is an
isometry ($\|\varphi(v)\| = \|v\|_x$).
\end{definition}

For vector bundles $\xi_1$, $\xi_2$ over the same base space one
can define $\xi=\xi_1 \oplus \xi_2$ by setting  $\xi_x = \xi_{1,x}
\oplus \xi_{2,x}$. In the following, points in $\xi_1 \oplus \xi_2$
will be denoted by a triple $(z,x_1,x_2)$, where $z \in
B\xi_{1}=B\xi_{2}$, $x_1 \in \xi_{1,z}$ and $x_2 \in \xi_{2,z}$.
 If $\xi_1$ and $\xi_2$ are both Banach bundles,
then $\xi_i \oplus \xi_2$ is also a Banach vector bundle with the
norm on   $\xi_{1,z}\oplus \xi_{2,z}$   defined  by
$\|(x_1,x_2)\|_z=\max (\|x_1\|_z, \|x_2\|_z)$. We will always use
this norm on $\xi_1 \oplus \xi_2$. We will also always assume that
the atlas on bundle $\xi_1 \oplus \xi_2$ respects this structure,
namely if $(\eta,U)$ is a bundle chart for $\xi_1 \oplus \xi_2$,
then its restriction (obtained through projection) to $\xi_i$ is
also a bundle chart for $\xi_i$ for $i=1,2$.

\section{Central hyperbolic sets, covering relations and cones} \label{sec:ch-sets}
In this section we introduce the setup for our approach. First we introduce the concept of a central
hyperbolic set (ch-set). This set will play the role of a compact
region in which we expect to find an  invariant
manifold. We will consider maps defined on ch-sets which have
certain contraction and expansion properties. On a ch-set we will have three directions:
  the direction of topological expansion, the direction of topological contraction, and a third,
  the central direction,  in which the dynamics will be weaker than in the first two directions.
  The contraction and expansion properties will be expressed in terms of covering relations.
  The ch-sets will be equipped with cones. The cones will help us to narrow down and specify
  in more detail the directions of expansion and contraction.

\subsection{Central hyperbolic sets}
\label{sec:chs-cvol}
Before we introduce the definition of a central hyperbolic set we will need the following definition.
\begin{definition}
 Let $\xi_u$ and $\xi_s$ be  Banach vector bundles, with
the base space $\Lambda$. Let $Z \subset \Lambda$. For $r>0$ we
define set $N(Z,r), N^\pm(Z,r) \subset \xi_u \oplus \xi_s$ by
\begin{eqnarray*}
  N(Z,r)=\{ (z,z_u,z_s) \in \xi_u \oplus \xi_s \ | \  z \in Z, \quad  \| z_u \| \leq r, \quad \| z_s \| \leq r
  \}, \\
  N^+(Z,r)=\{ (z,z_u,z_s) \in \xi_u \oplus \xi_s \ | \  z \in Z, \quad  \| z_u \| \leq r, \quad \| z_s \| = r
  \}, \\
 N^-(Z,r)=\{ (z,z_u,z_s) \in \xi_u \oplus \xi_s \ | \  z \in Z, \quad  \| z_u \| = r, \quad \| z_s \| \leq r
  \}
\end{eqnarray*}
If $r=1$, then we will write $N(Z), N^\pm(Z)$ instead of $N(Z,1),
N^\pm(Z,1) $ respectively.

\end{definition}

\begin{definition}
\label{def:cen-hyp} Assume that, we have  Banach vector bundles
$\xi_u$, $\xi_s$,  $\xi=\xi_u \oplus \xi_s$. Let $u$ and $s$ be
the fiber dimension of $\xi_u$ and $\xi_s$, respectively. Let  the
base space for $\xi$, denoted by  $\Lambda$, be a compact manifold
without boundary of dimension $c$.

 A \emph{central-hyperbolic set} $D$ (a ch-set),  is an object
consisting of the following data

\begin{enumerate}
\item $|D|$ a compact subset of $E\xi$.

\item a homeomorphism $\phi: \xi_u \oplus \xi_s \to \xi_u \oplus \xi_s$ such that
\begin{eqnarray}
  p_\xi (\phi(z)) &= & p_\xi(z) \quad \mbox{for $z \in E\xi$}, \\
 \phi(|D|) &= & N(\Lambda).
\end{eqnarray}

\end{enumerate}
\end{definition}

We will usually drop vertical  bars in the symbol $|D|$ and write
$D$ instead.

The notation $u,s$ and $c$ for the dimensions in Definition \ref{def:cen-hyp}, stand for the unstable,
 stable and central directions respectively. Their roles will become apparent once we introduce
 the definitions of covering relations and cone conditions for maps.

 For a central-hyperbolic set $D$
we define
\begin{eqnarray*}
D_{\phi} &=&  \phi(|D|)=N(\Lambda), \\
D_{\phi}^{-} &=& \{ (x,x_u,x_s) \in D_\phi, x \in \Lambda, x_u
\in\xi_{u,x},  x_s  \in \xi_{s,x} \  | \ \| x_u \|=1
\}, \\
D_{\phi}^{+} &=& \{ (x,x_u,x_s) \in D_\phi, x \in \Lambda, x_u \in
\xi_{u,x},  x_s  \in \xi_{s,x} \  | \ \| x_s \|=1 \}, \\
D^{-} &=& \phi^{-1}(D_{\phi}^{-}),\quad
D^{+}=\phi^{-1}(D_{\phi}^{+}).
\end{eqnarray*}

The sets $D^+$ and $D^-$ will later on be the entrance and exit
sets for a map defined on $D$, respectively.

From now on we assume that we work in the Banach vector bundle
$\xi_u \oplus \xi_s$ with the base space given by a compact
manifold without boundary $\Lambda$. We will now define an atlas on $\Lambda$,
 which will later be used for the definitions of covering relation and cone conditions.

%\begin{definition}
%Let $I$ be a finite set and  assume that we have an atlas
%$\{(\eta_i,U_i)\}_{i\in I}$ for our manifold $\Lambda$,
%\begin{equation}
% \eta_{i}:  U_i \to \eta_i(U_i) \subset \mathbb{R}^c, \quad \mbox{where $U_i \subset \Lambda$ is
%open. }  \label{eq:atlas}
%\end{equation}
%Moreover, we assume that for all $i \in I$ there exists
%$(\varphi,W) \in \Phi\xi$, such that $U_i \subset W$. Therefore on
%$p^{-1}(U_i)$ we have a chart map
%\begin{equation}
%\tilde{\eta}_i: p^{-1}(U_i) \ni z \mapsto (\eta_i(p(z)), \pi_2
%\varphi(z)) \in \eta_i(U_i) \times \left(\mathbb{R}^u \oplus
%\mathbb{R}^s\right).
%\end{equation}
% We will call it a  \emph{good atlas}.
%\end{definition}

\begin{definition}
For our vector bundle we define the \emph{ good atlas } as
follows:

Assume that we have an atlas $\{(\eta_i,U_i)\}_{i\in I}$  for
$\Lambda$, where $I$ is a finite set  and
\begin{equation}
 \eta_{i}:  U_i \to \eta_i(U_i) \subset \mathbb{R}^c, \quad \mbox{where $U_i \subset \Lambda$ is
open. }  \label{eq:atlas}
\end{equation}
Moreover, we assume that for all $i \in I$ there exists
$(\varphi,W) \in \Phi\xi$, such that $U_i \subset W$. We
fix one such $\varphi$ for each $i$. Therefore on $p^{-1}(U_i)$
we have a chart map
\begin{equation}
\tilde{\eta}_i: p^{-1}(U_i) \ni z \mapsto (\eta_i(p(z)), \pi_2
\varphi(z)) \in \eta_i(U_i) \times \left(\mathbb{R}^u \oplus
\mathbb{R}^s\right).
\end{equation}
\end{definition}

Therefore the atlas is good if the domains of the chart maps on
$\Lambda$ are also domains for some bundle chart maps.
 From now we will always implicitly assume that we work with a good atlas.

\begin{definition}
Let $Z \subset U_i$, where $(\eta_i,U_i)$ is a chart from some
good atlas. For $r >0$ we define $N_{\eta_i}(Z,r)$, the normal
neighborhood,  and $N^\pm_{\eta_i}(Z,r)$ normal exit and entry
sets,  as follows
\begin{eqnarray}
  N_{\eta_i}(Z,r)&=&\eta_i(Z) \times \overline{B}_u(0,r) \times
  \overline{B}_s(0,r) \subset \mathbb{R}^c \times \mathbb{R}^u \times \mathbb{R}^s,  \\
  N_{\eta_i}^-(Z,r)&=&\eta_i(Z) \times \partial B_u(0,r) \times
  \overline{B}_s(0,r) \subset   N_{\eta_i}(Z,r),
   \\
  N_{\eta_i}^+(Z,r)&=&\eta_i(Z) \times \overline{B}_u(0,r) \times
   \partial B_s(0,r) \subset   N_{\eta_i}(Z,r).
\end{eqnarray}
If $r=1$ then we will often drop it and write $N_{\eta_i}^\pm(Z)$
instead.
\end{definition}
Obviously, for $Z \subset U_i$ we have
\begin{displaymath}
  N(Z,r)=\tilde{\eta}_i^{-1}\left(N_{\eta_i}(Z,r)\right),\quad
  N^\pm(Z,r)=\tilde{\eta}_i^{-1}\left(N^\pm_{\eta_i}(Z,r)\right) \quad .
\end{displaymath}

Let $D$ be ch-set and let
\[
f:|D|\rightarrow \xi_u \oplus \xi_s
\]
be a continuous map. We  define
\[
f_{\phi}:=\phi\circ f\circ\phi^{-1}: N(\Lambda) \rightarrow \xi_u
\oplus \xi_s.
\]
When using local coordinates around $z$ and $f(z)$, given by
charts  $(\eta_j,U_j)$ and $(\eta_i,U_i)$ respectively, we will
consider functions $f_{ij}$ given by
\begin{equation}
f_{ij}:=\tilde{\eta}_{i}\circ f_{\phi}\circ \tilde{\eta}_{j}^{-1}.
\label{eq:f_x}
\end{equation}

The following commutative diagram illustrates the meaning and mutual
relations between maps $f$, $f_{\phi}$ and $f_{ij}$.

\[
\begin{array}{ccc}
D & \overset{f}{\longrightarrow} & \xi_u \oplus \xi_s \\
\downarrow\phi &  & \downarrow\phi \\
N(\Lambda) & \overset{f_{\phi}}{
\longrightarrow} & \xi_u \oplus \xi_s \\
\uparrow \tilde{\eta}_{j}^{-1} &  & \quad\quad\uparrow
\tilde{\eta}_{i}^{-1} \\
\eta_j( U_j)\times\overline{B}_{u}\times\overline{B}_{s} &
\overset{f_{ij}}{\longrightarrow} & \eta_i(U_i)\times\mathbb{R}^{u+s}%
\end{array}
\]

\subsection{Covering relations for ch-sets}
In this section we give the definition of a covering relation for a ch-set. Covering relations will be used
later on in the proof of our main result to ensure that we have an invariant set in the interior of a ch-set.
\begin{figure}
[ptb]
\begin{center}
\includegraphics[height=6cm]{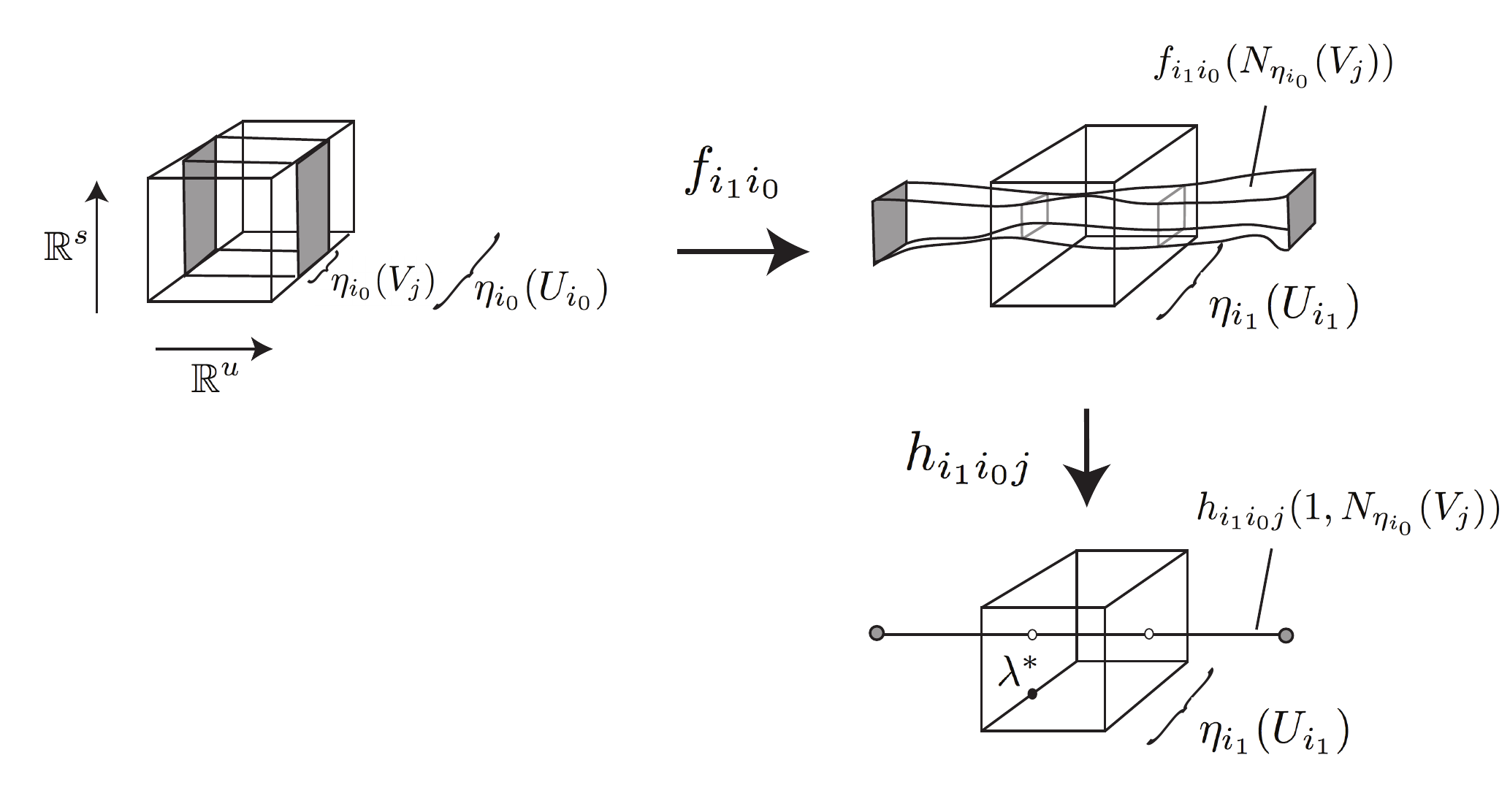}
\caption{ Covering relation for a ch-set. }%
\label{fig:cover}%
\end{center}
\end{figure}

\begin{definition}
\label{def:f-covers} Assume that we have two Banach vector bundles
$\xi_u$ and $\xi_s$ with the same base space $\Lambda$, a compact
manifold without boundary of dimension $c$.

 Let $D \subset \xi_u \oplus \xi_s$ be a
ch-set and a map $f:D \to \xi_u \oplus \xi_s$ be continuous.

Assume that  $\{(\eta_i,U_i)\}_{i \in I}$ is a good atlas  on
$\Lambda$, where $I$ is some finite set, and $\{V_j\}_{j \in J}$
($J$ a finite set) is an open covering of $\Lambda$, such that for
every $j \in J$ there exists $i_0, i_1 \in I$ (not necessarily
unique) such that
\begin{eqnarray}
   V_j &\subset& U_{i_0}  \label{eq:VinU} \\
   f(p^{-1}(V_j) \cap |D|) &\subset& p^{-1}(U_{i_1}).
    \label{eq:fgoodatals}
\end{eqnarray}

We will say that  ch-set $D$ \emph{$f$-covers itself}, which we
will denote by
\[
D\overset{f}{\Longrightarrow}D,
\]
if the following conditions are satisfied for every $V_j$.

Let  $U_{i_0}$ and $U_{i_1}$ by any sets  such that conditions
(\ref{eq:VinU}) and (\ref{eq:fgoodatals}) hold. The conditions
are:
\begin{enumerate}
\item There exists a continuous homotopy
\begin{displaymath}
h_{i_1i_0j}:[0,1]\times N_{\eta_{i_0}}(V_j) \rightarrow
\eta_{i_1}(U_{i_1}) \times \mathbb{R}^{u}\times\mathbb{R}^{s}
\end{displaymath}
 such
that the following conditions hold true
\begin{eqnarray}
h_{i_1i_0j}(0,z)  = f_{i_1 i_0}(z) &  & \mbox{for $z \in N_{\eta_{i_0}}(V_j)$,}\\
h_{i_1i_0j}\left([0,1], N_{\eta_{i0}}^-(V_j)\right) \cap
N_{\eta_{i_1}}(U_{i_1}) & =&\emptyset, \label{eq:f-homotopy-cond}
\\
h_{i_1i_0j}\left([0,1],N_{\eta_{i_0}}(V_{j})\right)\cap
N^+_{\eta_{i_1}}(U_{i_1}) & =&\emptyset.
\end{eqnarray}

\item If $u>0,$ then there exists a linear map $A_{i_1i_0j}:\mathbb{R}^{u}\rightarrow
\mathbb{R}^{u}$ and $\lambda^* \in \eta_{i_1}(U_{i_1})$ such that
\begin{align}
h_{i_1i_0j}(1,(\lambda,x,y)) & =(\lambda^*,A_{i_1i_0j}x,0),  \nonumber \\
A_{i_1i_0j}(\partial B_{u}(0,1)) & \subset\mathbb{R}^{u}\backslash\overline{B_{u}}%
(0,1).   \label{eq:f-homotopy-cond2}
\end{align}

\item If $u=0$, then there exists $\lambda^* \in \eta_{i_1}(U_{i_1})$ such that
\[
h_{i_1i_0j}(1,(\lambda,y))=(\lambda^*,0).
\]
\end{enumerate}
\end{definition}

The idea behind the definition can be summarized as follows. In local coordinates the map
topologically expands the ch-set in the unstable direction and contracts it in the stable direction.
 In the central direction we do not require any contraction or expansion properties. We just require
 that the local maps are properly defined on sets $N_{\eta_{i_0}}(V_j)$ (see Figure \ref{fig:cover}).

\begin{remark}
In fact it is enough to require that for a given $j$  there exists
a required homotopy for just one pair of  $U_{i_0}$ and $U_{i_1}$
satisfying  conditions (\ref{eq:VinU}) and (\ref{eq:fgoodatals}).
For the other pairs the homotopy can be obtained using the
transition maps from the atlas.
\end{remark}

%\textbf{MC: tu nowa definicja i komentarz po niej.}
%
%\begin{definition}
%\label{def:f-back-covers}
% Let $D \subset \xi_u \oplus \xi_s$ be a
%ch-set and a map $f:D \to \xi_u \oplus \xi_s$ be continuous. We will say that  ch-set $D$
% \emph{$f$-back covers itself}, which we
%will denote by
%\[
%D\overset{f}{\Longleftarrow}D,
%\]
%if $D$ $f$-covers itself, but with the roles of the coordinates $u$ and $s$ reversed. \end{definition}

%Later on we will work with an invertible map $f:U \to U$ with $D\subset U \subset \xi_u \oplus \xi_s$.
%We will consider the case where $D$ $f$-covers itself and $f^{-1}$-back covers itself. Such setting
% is natural since the roles of the stable and unstable directions for the forward and backward maps are reversed.

%TCIDATA{OutputFilter=latex2.dll}
%TCIDATA{Version=5.00.0.2606}
%TCIDATA{LaTeXparent=0,0,online-edit.tex}
\subsection{Cones, horizontal and vertical disks}
\label{sec:conec}
In this section we introduce horizontal and vertical discs. These will be later on used as building blocks to construct both the invariant manifold and its stable and unstable manifolds in Section \ref{sec:main-th}.  Equipping the ch-set with cones will allow us to consider horizontal discs which are appropriately flat and vertical discs which are appropriately steep.

In this section we  assume that we have a fixed
Banach vector bundle $\xi_u \oplus \xi_s$ with a base space
$\Lambda$, which is a compact manifold without boundary of
dimension $c$.

\begin{definition}
Let $Q:\mathbb{R}^n \to \mathbb{R}$ be a quadratic
form. We define positive and negative
cone by
\begin{eqnarray*}
  C^+(Q) = \{x \in \mathbb{R}^n \ | \  Q(x) >0\}, \\
  C^-(Q) = \{x \in \mathbb{R}^n \ | \  Q(x) <0\}.
\end{eqnarray*}
If $z \in \mathbb{R}^n$ then we define positive and negative cones centered at $z$ by
\begin{eqnarray*}
  C^+(Q,z)= \{x \in \mathbb{R}^n \ | \  Q(z-x) >0\} = z + C^+(Q), \\
  C^-(Q,z)= \{x \in \mathbb{R}^n \ | \  Q(z-x) <0\} = z + C^-(Q).
\end{eqnarray*}
\end{definition}

In the sequel we will work with cones given by locally defined
quadratic forms on $\xi_u \oplus \xi_s$ and the sets of interests
will be the intersection of positive or negative cones centered at
$z$ with $N(\Lambda)$. Let $Q:\mathbb{R}^c \times \mathbb{R}^u
\times \mathbb{R}^s \to \mathbb{R}$ be a quadratic form.  Let $z
\in \mathbb{R}^c \times \mathbb{R}^u \times \mathbb{R}^s $, we set
\begin{eqnarray*}
 C^{\pm}_r(Q,z) = C^{\pm}(Q,z) \cap \{ (\theta,x,y) \in
  \mathbb{R}^c \times \mathbb{R}^u \times \mathbb{R}^s
          \  | \ \|x\| \leq r, \ \|y\| \leq r \}.
\end{eqnarray*}

\begin{definition}
\label{def:ch-cones} Let $D \subset \xi_u \oplus \xi_s$ be a
ch-set.

Let $\{(\eta_i,U_i)\}_{i \in I}$ be a good atlas on $\Lambda$ and
assume that for all $i \in I$ we have quadratic forms
 $Q_{i,h},Q_{i,v}:\mathbb{R}^{c}\times\mathbb{R}
^{u}\times\mathbb{R}^{s}\to \mathbb{R}$  given by
\begin{eqnarray}
Q_{i,h}(\theta,x,y) &=& \alpha_i(x)-\beta_i(y)-\gamma_i(\theta),  \label{eq:cone-h} \\
Q_{i,v}(\theta,x,y) &=& \alpha_i(x)-\beta_i(y)+\gamma_i(\theta),
\label{eq:cone-v}
\end{eqnarray}
where $\alpha_i:\mathbb{R}^{u}\to \mathbb{R}$,
$\beta_i:\mathbb{R}^{s}\to \mathbb{R}$ and
$\gamma_i:\mathbb{R}^{c}\to \mathbb{R}$ are positive definite
quadratic forms. We assume that these forms are uniformly bounded
in the following sense: there exist $M_+,M_- \in \mathbb{R}_+$,
such that for every $i \in I$ hold
\begin{eqnarray}
   M_- \|x\|^2 \leq \alpha_i(x) \leq M_+ \|x\|^2 \qquad \mbox{for $x \in
   \mathbb{R}^u$,} \notag \\
    M_- \|y\|^2 \leq \beta_i(y) \leq M_+ \|y\|^2 \qquad \mbox{for $y \in
   \mathbb{R}^s$,} \label{eq:M-est} \\
    M_- \|\theta\|^2 \leq \gamma_i(\theta) \leq M_+ \|\theta\|^2 \qquad \mbox{for $\theta \in
   \mathbb{R}^c$}. \notag
\end{eqnarray}

Assume that there exists an open covering $\{V_j\}_{j \in J}$ of
$\Lambda$ inscribed in $\{U_i\}_{i\in I}$, such that  $V_j$ is
contractible for each $j$ and the following two conditions are
satisfied

\begin{enumerate}

\item \label{cond:cones-chart}
 for any point $q \in |D|$ there exists $j \in J$ and $i\in I$ such that $
p(q) \in V_j$, $V_j \subset U_i$
and
\begin{eqnarray}
% \pi_1\left(  \overline{C^+_1(Q_{i,h},\tilde{\eta}_i(q))} \right) \subset \eta_i(V_j),
\overline{C^+_1(Q_{i,h},\tilde{\eta}_i(\phi(q)))} \subset
N_{\eta_i}(V_j),
\label{eq:atls-cond} \\
% \pi_1 \left(   \overline{C^-_1(Q_{i,v},\tilde{\eta}_i(q))} \right)   \subset \eta_i(V_j).
\overline{C^-_1(Q_{i,v},\tilde{\eta}_i(\phi(q)))} \subset
N_{\eta_i}(V_j).
 \label{eq:atls-cond2}
\end{eqnarray}

\item \label{cond:chart-pair}
for any $j \in J$ there exists an $i \in I$, such that for all $q
\in |D| \cap p^{-1}(V_j)$ there exists a $V_{j(q)}\subset U_i$
such that
\begin{eqnarray}
 %\pi_1\left( \overline{C^+_1(Q_{i,h},\tilde{\eta}_i(q))} \right) \subset \eta_i(U_i),
    \overline{C^+_1(Q_{i,h},\tilde{\eta}_i(\phi(q)))} \subset N_{\eta_i}(V_{j(q)}) \subset
    N_{\eta_i}(U_i),
\label{eq:atls-good-pair} \\
%  \pi_1 \left( \overline{C^-_1(Q_{i,v},\tilde{\eta}_i(q))} \right)  \subset \eta_i(U_i).
\overline{C^-_1(Q_{i,v},\tilde{\eta}_i(\phi(q)))} \subset N_{\eta_i}(V_{j(q)}) \subset
N_{\eta_i}(U_i). \label{eq:atls-good-pair2}
\end{eqnarray}
\end{enumerate}

 The structure consisting
of $(D,\{\eta_i,U_i,Q_{i,h},Q_{i,v}\}_{i \in I},\{V_j\}_{j \in
J})$ will be called a \emph{ch set with cones}. Usually we will
refer to this structure by $D$.
\end{definition}
\begin{remark}
In fact we may allow to have different $\gamma_i$ in $Q_{i,h}$ and
$Q_{i,h}$, but $\alpha_i$ and $\beta_i$ must coincide.
\end{remark}

\begin{definition}
\label{def:cone-charts} Assume that $D$ is a ch-set with cones.

Let $q \in |D|$. We will say that $(V_j,U_i)$ is a
\emph{cone enclosing pair for $q$}, when $V_j$ and $U_i$ are as in point \ref{cond:cones-chart}. of
Def.~\ref{def:ch-cones}.

We will say that a pair $(V_j,U_i)$ is a \emph{cones chart pair}
iff condition \ref{cond:chart-pair} in
Def.~\ref{def:ch-cones}.
is satisfied.

\end{definition}

\begin{example} \label{ex:sets} Let us choose some $k\in \mathbb{N}$ such that $k\ge 9$.
 Consider a vector bundle $(\mathbb{R}/k) \times
\mathbb{R}^u \times \mathbb{R}^s = S^1\times \mathbb{R}^u \times
\mathbb{R}^s$. Let $D=N(\mathbb{R}/k)$ and let $\alpha(x)=x^2$,
$\beta(y)=y^2$, $\gamma(\theta)=\theta^2$, $Q_h=\alpha - \beta -
\gamma$ and $Q_v=\alpha - \beta + \gamma$. If for
$i,j\in\{0,\ldots ,k\}$ we define the sets 
\begin{eqnarray}
V_j &=&j+(0,5) \textrm{ mod } k, \notag \\
U_i&=&i+(0,9) \textrm{ mod } k, \notag
\end{eqnarray}
then for any point $q$ for which $\pi_{\theta}(q)\in[i,i+1]$ the pair $(V_{i-2},U_{i-4})$
is both a cones chart pair, as well as a cone enclosing pair for $q$ (see Figure \ref{fig:1}).

\end{example}
%%%%%%%%%%%
\begin{figure}
[ptb]
\begin{center}
\includegraphics[height=6cm]{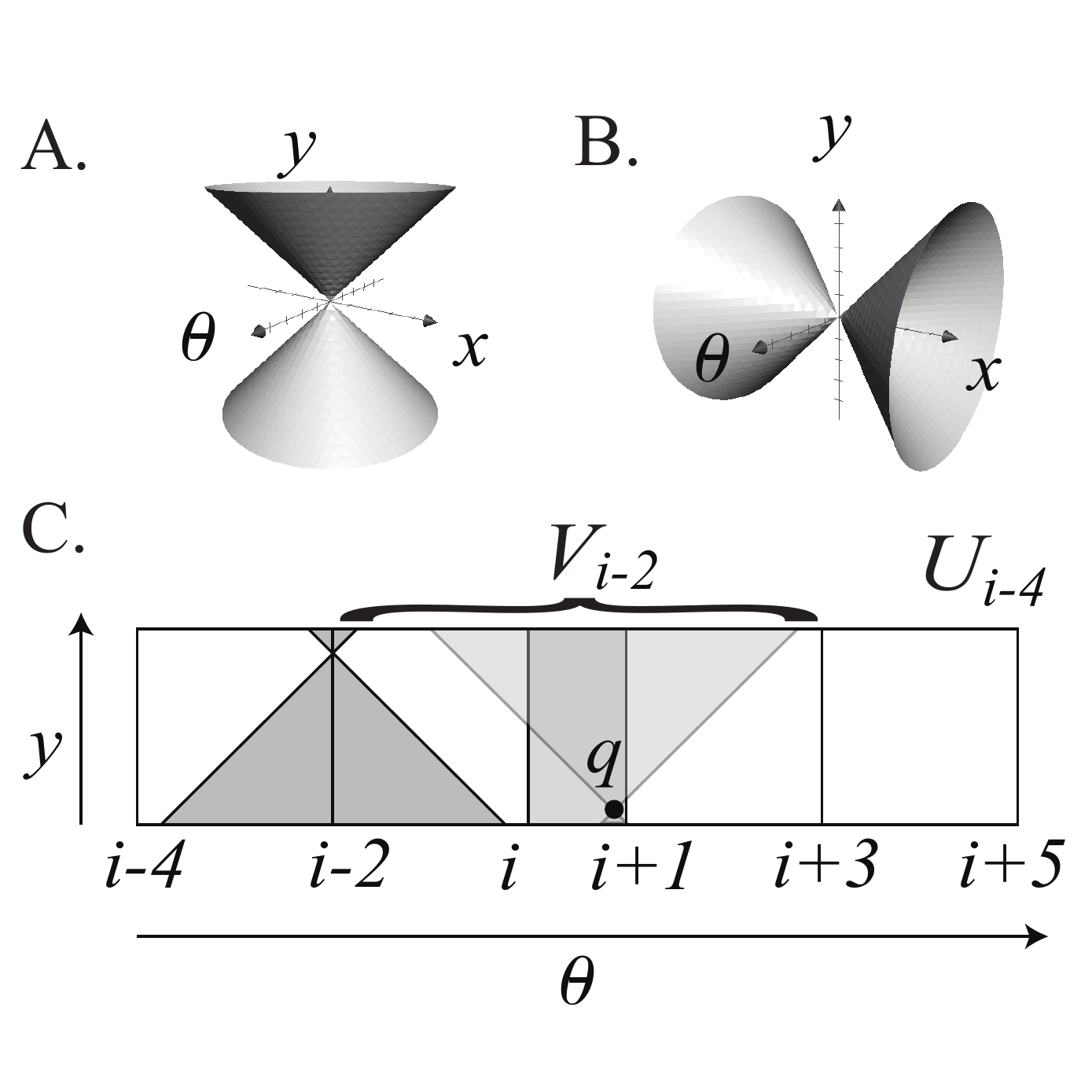}
\caption{A vertical cone (A) and a horizontal cone (B), together with a projection of the cones chart pair $(V_{i-2},U_{i-4})$ and two vertical cones (one centered at $q$ and the other at a point from $|D| \cap p^{-1}(V_{i-2})$) onto the $\theta, y$ coordinates   (C). }%
\label{fig:1}%
\end{center}
\end{figure}
%%%%%%%

We now move on to the definitions of horizontal and vertical discs.
\begin{definition}
\label{def:hor-disk} Let $D$ be a ch-set.
We say that $b:\overline{B}_{u}\rightarrow D$ is a \emph{ horizontal disc} in $D$ if there
exists $U_i$ and a homotopy
$h:[0,1]\times\overline{B}_{u}\rightarrow N_{\eta_i}(U_i)$ such
that
\begin{eqnarray}
h(0,x) &=& \tilde{\eta}_i(b_{\phi}(x)),  \quad x \in \overline{B}_u,\nonumber \\
h(1,x) &=& (\lambda^{\ast},x,0), \quad \mbox{for all
$x\in\overline{B}_u$ and some
  $\lambda^{\ast}\in \eta_i(U_i)$, }  \nonumber \\
h(t,x) &\in& N_{\eta_i}^-(U_i),\quad\mbox{for all $t\in [0,1]$ and
$x\in\partial\overline{B}_{u}$},
   \label{eq:horiz-in-boundary}
\end{eqnarray}
where $b_\phi=\phi \circ b$. Let us set $|b|=b(\overline{B}_u)$.

If additionally, $|b| \subset Z \subset U_i$ and $h([0,1]\times
\overline{B}_u) \subset N_{\eta_i}(Z)$, then we will say that $b$
is a horizontal disk for the pair $(Z,U_i)$.
\end{definition}

\begin{definition}
\label{def:hor-cones} Let $D$ be a ch-set with cones, $Z \subset
U_i$ and let $b:\overline{B}_{u}\rightarrow D$ be a horizontal
disk for $(Z,U_i)$. We say that $b$ is a \emph{ horizontal disc
satisfying the cone condition} for $(Z,U_i)$ if the following condition
holds
\begin{eqnarray}
 Q_{i,h}( \tilde{\eta}_i(b_\phi(x_0)) -
\tilde{\eta}_i(b_\phi(x_1)) ) >  0,
          \quad \mbox{for all $x_0,x_1 \in \overline{B}_u$}.
\label{eq:cone-cond-h}
\end{eqnarray}

If condition (\ref{eq:cone-cond-h}) holds for all $U_j$, such that
$|b| \subset U_j$, then we say that $b$ is a horizontal disk in
$D$.
\end{definition}

The idea behind Definition \ref{def:hor-cones} is that when we consider
a horizontal disc in local coordinates and attach a horizontal cone to any
of its points, then the entire disc will be contained in the cone (see Figure \ref{fig:horizontal-cc}).

%%%%%%%%%%%
\begin{figure}
[ptb]
\begin{center}
\includegraphics[height=4cm]{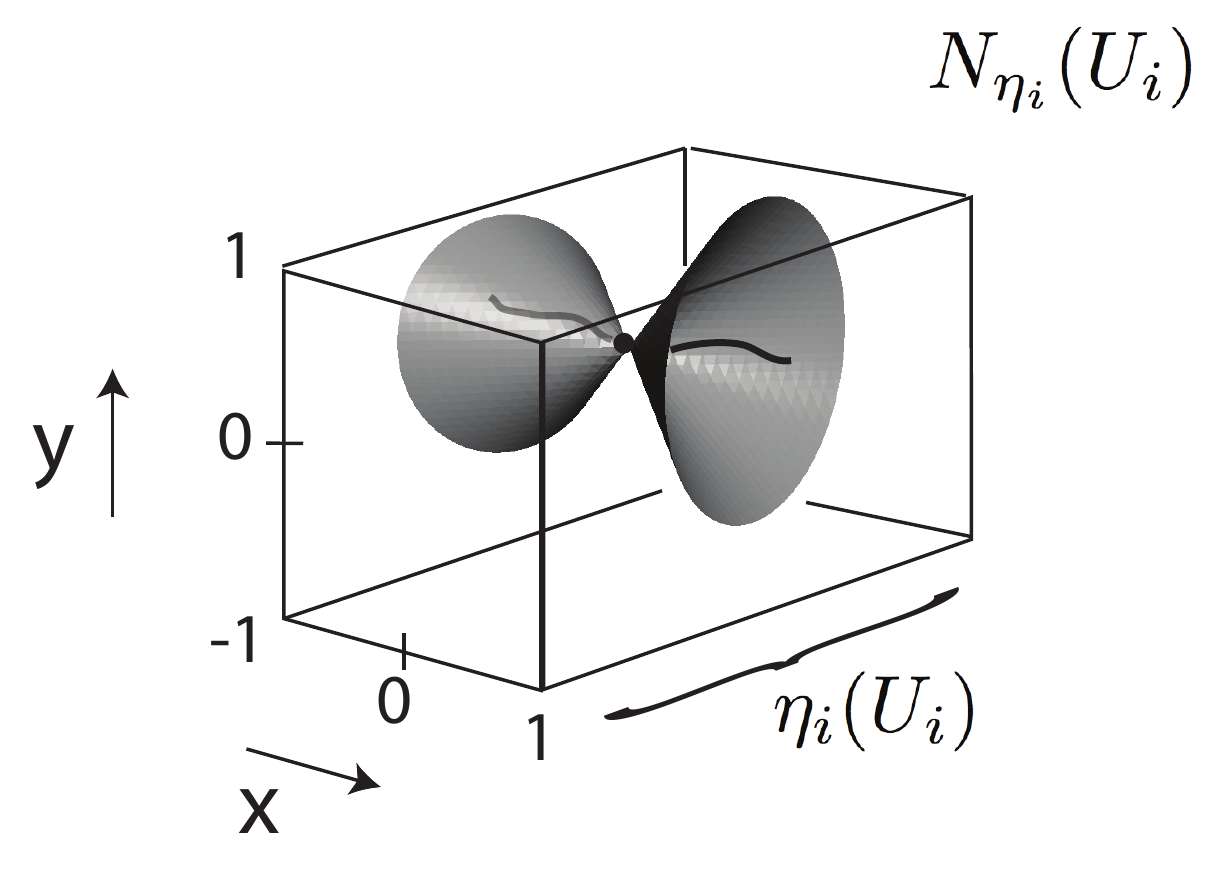}
\caption{A horizontal disc which satisfies cone conditions (here $u=s=1$).}%
\label{fig:horizontal-cc}%
\end{center}
\end{figure}
%%%%%%%

Analogously we define a vertical disk as follows.

\begin{definition}
\label{def:vert-disk} Let $D$ be a ch-set.
We say that $b:\overline{B}_{s}\rightarrow D$ is a \emph{ vertical disc} in $D$ if there
exists $U_i$ and a homotopy
$h:[0,1]\times\overline{B}_{s}\rightarrow N_{\eta_i}(U_i)$ such
that
\begin{eqnarray}
h(0,y) &=& \tilde{\eta}_i(b_{\phi}(y)),  \quad y \in \overline{B}_s, \nonumber \\
h(1,y) &=& (\lambda^{\ast},0,y), \quad \mbox{for all
$y\in\overline{B}_s$ and some
  $\lambda^{\ast}\in \eta_i(U_i),$ }  \nonumber \\
h(t,y) &\in& N_{\eta_i}^+(U_i),\quad\mbox{for all $t\in [0,1]$ and
$y\in\partial\overline{B}_{s}$},
   \label{eq:vert-in-boundary}
\end{eqnarray}
where $b_\phi=\phi \circ b$. Let us set $|b|=b(\overline{B}_s)$.

If additionally, $|b| \subset Z \subset U_i$ and $h([0,1]\times
\overline{B}_s) \subset N_{\eta_i}(Z)$, then we will say that $b$
is a vertical disk for $(Z,U_i)$.
\end{definition}

\begin{definition}
\label{def:vert-cones} Let $D$ be a ch-set with cones, $Z \subset
U_i$ and let $b:\overline{B}_{s}\rightarrow D$ be a vertical
disk for $(Z,U_i)$. We say that $b$ is a \emph{ vertical disc
satisfying the cone condition} for $(Z,U_i)$ if the following condition
holds
\begin{eqnarray}
 Q_{i,v}( \tilde{\eta}_i(b_\phi(y_0)) -
\tilde{\eta}_i(b_\phi(y_1)) ) < 0,
          \quad \mbox{for all $y_0,y_1 \in \overline{B}_s$}.
\label{eq:cone-cond-v}
\end{eqnarray}

If condition (\ref{eq:cone-cond-v}) holds for all $U_j$, such that
$|b| \subset U_j$, then we say that $b$ is a vertical disk in
$D$.
\end{definition}

The following lemma was proved  in a slightly different setting in
\cite[Lemma 5]{Z}
\begin{lemma}
\label{lem:hord-is-graph} Assume that $b:\overline{B}_u \to D$ is
a horizontal disk satisfying cone conditions in $U_i$.

Then $b$ can be represented as the graph of the function in the
following sense: there exists  continuous functions
$d_c:\overline{B}_u \to U_i$, $d_y:\overline{B}_u \to
\overline{B}_s$ such that for any $x_1 \in \overline{B}_u$ there
exists a uniquely defined $x \in \overline{B}_u$, such that
\begin{equation}
  \tilde{\eta}_i(b_\phi(x_1))=(d_c(x),x,d_y(x)).
\end{equation}
\end{lemma}
An analogous lemma is valid for vertical disks
satisfying cone conditions.

The lemma below shows that  for a graph of a function over "unstable
coordinate" with its range in $p^{-1}(V_j)$ for some $V_j$ is a
horizontal disk in $V_j$.
\begin{lemma}
\label{lem:graph-is_hord}  Assume that $V_j\subset \Lambda$ is
contractible and that we have  continuous functions
$d_c:\overline{B}_u \to V_j$, $d_y:\overline{B}_u \to
\overline{B}_s$. Let $d:\overline{B}_u \to N_{\eta_i}(V_j)$ be
given by $d(x)=(d_c(x),x,d_y(x))$.

Then there exists  homotopy
$h:[0,1]\times\overline{B}_{u}\rightarrow N_{\eta_i}(V_j)$ such
that
\begin{eqnarray}
h(0,x) &=& d(x),  \quad x \in \overline{B}_u, \nonumber \\
h(1,x) &=& (\lambda^{\ast},x,0), \quad \mbox{for all
$x\in\overline{B}_u$ and some
  $\lambda^{\ast}\in \eta_i(V_j),$ }  \nonumber \\
h(t,x) &\in& N_{\eta_i}^-(V_j),\quad\mbox{for all $t\in [0,1]$ and
$x\in\partial\overline{B}_{u}$},
   \label{eq:graph-h-in-bd}
\end{eqnarray}
which means that $\phi^{-1}\circ \tilde \eta_i^{-1} \circ d:
\overline{B}_u \to D $ is horizontal disc in $V_j$.
\end{lemma}
\begin{proof}
 From the
contractibility of $V_j$ it follows that there exists $\lambda^* \in V_j$
and a continuous homotopy $g:[0,1] \times V_j \to V_j$ such that
\begin{equation*}
  g(0,\lambda)=\lambda, \quad g(1,\lambda)=\lambda^*, \qquad \mbox{for all $\lambda \in V_j.$}
\end{equation*}
We define homotopy $h$ by setting
\begin{equation}
  h(t,x)=(g(t,d_c(x)),x,(1-t)d_y(x)).
\end{equation}
\end{proof}

An analogous lemma is valid for vertical disks.

%TCIDATA{OutputFilter=latex2.dll}
%TCIDATA{Version=5.00.0.2606}
%TCIDATA{LaTeXparent=0,0,online-edit.tex}

\section{Cone-conditions for maps and main theorems}
\label{sec:main-th}

Our main result is given in Theorem \ref{th:main}. There we will show that
 for a map which satisfies cone conditions
 (see Def.~\ref{def:f-cone-cond}) there exists an invariant manifold
 together with its stable and unstable manifolds inside of the ch-set. For the proof
 of Theorem \ref{th:main} we will first show that if a map satisfies cone conditions
 then an image of a horizontal disc which satisfies cone conditions is also a horizontal
 disc which satisfies cone conditions. This fact will be used to construct vertical discs
 of points whose forward iterations remain inside of the ch-set. Using these discs we will
  construct our invariant manifold together with its stable and unstable manifolds.

We start by defining cone conditions for a map.

\begin{definition}
\label{def:f-cone-cond}Assume that $D$ is a central-hyperbolic set
with cones. Let $f:D \to\xi_{u} \oplus\xi_{s}$ be continuous and
assume that $D\overset {f}{\Longrightarrow}D$. Moreover, we assume
that the inscribed coverings \textbf{$\{V_j\}_j\in J$} used
in covering relation $D\overset{f}{\Longrightarrow}D$ (Def.~\ref{def:f-covers}%
) and cones in $D$ (Def.~\ref{def:ch-cones}) coincide.

We say that $f$ \emph{satisfies (forward) cone conditions on $D$} if there
exists an $m>1$ such that:

For any $x\in D$ and any $(V_{j_{0}}, U_{i_{0}})$ which is a cone enclosing
pair for $x$, if $f(x)\in D$ then there exist $i_{1}\in I$, $j_{1}\in J$ such
that $(V_{j_{1}}, U_{i_{1}})$ is a cone enclosing pair for $f(x)$, and
$f(p^{-1}(V_{j_{0}})\cap| D| ) \subset p^{-1}(U_{i_{1}})$; what is more, for
any $x_{1},x_{2}\in N_{\eta_{i_{0}}} (V_{j_{0}})$ such that
\begin{equation}
Q_{i_{0},h}(x_{1}-x_{2}) \geq 0  \label{eq:ccQwej}
\end{equation}
we have
\begin{equation}
Q_{i_{1},h}(f_{i_{1} i_{0}}(x_{1})-f_{i_{1} i_{0}}(x_{2})) > m Q_{i_{0}%
,h}(x_{1}-x_{2}). \label{eq:cone-cond-Q}%
\end{equation}

\end{definition}

We will now define backward cone conditions
for the inverse map. In order to do so we will need the following
notations.
Let $Q_{i,h},Q_{i,v}:\mathbb{R}^{c}\times\mathbb{R}^{u}\times\mathbb{R}%
^{s}\rightarrow\mathbb{R}$ be defined as in (\ref{eq:cone-h}) and
(\ref{eq:cone-v}). We define $Q_{i,h}^{-},Q_{i,v}^{-}:\mathbb{R}^{c}%
\times\mathbb{R}^{u}\times\mathbb{R}^{s}\rightarrow\mathbb{R}$ as
\begin{align}
Q_{i,h}^{inv}(\theta,x,y)  &  :=-\alpha_{i}(x)+\beta_{i}(y)-\gamma_{i}%
(\theta),\label{eq:cone-h-inv}\\
Q_{i,v}^{inv}(\theta,x,y)  &
:=-\alpha_{i}(x)+\beta_{i}(y)+\gamma_{i}(\theta).
\label{eq:cone-v-inv}%
\end{align}

\begin{definition}
Assume that $(D,\{\eta_{i},U_{i},Q_{i,h},Q_{i,v}\}_{i\in
I},\{V_{j}\}_{j\in J})$ is a ch-set with cones.  Let
$f:\xi_{u}\oplus\xi_{s} \supset \dom(f) \rightarrow\xi
_{u}\oplus\xi_{s}$ be a homeomorphism, such that $D \subset \dom(f^{-1})$. Assume that $D\overset{f^{-1}%
}{\Longrightarrow}D,$ with reversed roles of the coordinates $u$ and $s$. We
say that $f$ \emph{satisfies backward cone conditions }if $f^{-1}$ satisfies
forward cone conditions on $(D,\{\eta_{i},U_{i},Q_{i,h}^{-inv},$ $Q_{i,v}%
^{inv}\}_{i\in I},$$\{V_{j}\}_{j\in J})$, with reversed roles of
the coordinates $u$ and $s$.
\end{definition}

We will now show that an intersection of the ch-set with an image of a horizontal disc which
 satisfies cone conditions is a horizontal disc which satisfies cone conditions.

\begin{lemma}
\label{lem:im-of-hor} Assume that
\[
D\overset{f}{\Longrightarrow}D,
\]
and that $f$ satisfies (forward) cone conditions. Let $b:\overline{B}%
_{u}\rightarrow D$. Assume that for any $x\in B^{u}$ we have a $(V_{j(x)}%
,U_{i(x)})$ which is a cone enclosing pair for $b(x)$, and that $b$ is a
horizontal disc which satisfies cone condition for $(V_{j(x)},U_{i(x)})$.
Then
\[
\beta=\{q\in D|q=f(b(x))\text{ \textit{for some }}x\in B^{u}\}
\]
is nonempty and for any $q\in\beta$ there exists a $(V_{j(q)},U_{i(q)})$ which
is a cone enclosing pair for $q$, such that $\beta$ is a horizontal disc which
satisfies cone condition for $(V_{j(q)},U_{i(q)})$.
\end{lemma}

\begin{proof}
Without any loss of generality we can assume that $\phi=\Id$ and therefore
$D_{\phi}=D=N(\Lambda)$, $b_{\phi}=b$ and $f_{\phi}=f$.

 Take any $V_{j},U_{i_{0}}, U_{i_{1}}$ such that
$b(\overline{B}_{u})\subset p^{-1}(V_{j})$, $V_{j}\subset
U_{i_{0}}$ and
\begin{equation}
f(p^{-1}(V_{j})\cap| D| )\subset p^{-1}(U_{i_{1}}). \label{eq:coincide-j(x)}%
\end{equation}
The existence of such sets is a direct consequence of
Def.~\ref{def:f-covers}. From Def.~\ref{def:hor-cones} it follows
that there exists a homotopy $h:[0,1] \times\overline{B}_{u} \to
N_{\eta_{i_{0}}}(V_{j})$ and $\lambda^{*} \in \eta_{i_{0}}(V_{j})$
satisfying
\begin{align}
h(0,x)  &  =\eta_{i_{0}}(b(x)), \quad x \in\overline{B}_{u},\\
h(1,x)  &  =(\lambda^{*},x,0), \quad x \in\overline{B}_{u},\\
h(t,\partial B_{u})  &  \subset N^{-}_{\eta_{i_{0}}}(V_{j}), \quad t \in[0,1].
\end{align}
Let us denote by $g$ the homotopy $g=h_{i_{1}i_{0} j}$ appearing in
Def.~\ref{def:f-covers} and also let $A=A_{i_{1}i_{0} j}$. To study the set
$f(b(\overline{B}_{u}))$ and its behavior under the homotopy $h$ and $g$ it is
enough to use charts $(\eta_{i_{1}},U_{i_{1}})$ in the range and $(\eta
_{i_{0}},U_{i_{0}})$ in the domain.

We will show now that for any $z \in\overline{B}_{u}$ there exits $x \in
B_{u}$, such that
\begin{equation}
\pi_{x} f_{i_{1} i_{0}}( \tilde{\eta}_{i_{0}}( b(x)))= z. \label{eq:x-of-z}%
\end{equation}
Later we will prove the uniqueness of such $x$. This will allow us to define
map $z \mapsto f(b(x(z))) $, which we will show to be the horizontal disc we
are looking for.

To study equation (\ref{eq:x-of-z}) we consider a parameter dependent map
$H:[0,1]\times\overline{B}_{u} \to\mathbb{R}^{u}$ given by
\begin{equation}
H(t,x)= \pi_{x} g(t, h(t,x) ) - (1-t)z. \label{eq:hom-x-of-z}%
\end{equation}
Using the local Brouwer degree we will study the parameter dependent equation
\begin{equation}
H_{t}(x)=0. \label{eq:Hhor}%
\end{equation}
Any solution of (\ref{eq:Hhor}) with $t=0$ is a solution of problem
(\ref{eq:x-of-z}) and vice versa.

Observe first that the local Brouwer degree (see Section
\ref{sec:deg-prop} for the properties of the degree) of $H_{t}$ on
$B_{u}$ at $0$, denoted by $\deg(H_{t},B_{u},0)$, is defined
because from Def.~\ref{def:f-covers} it follows that
\begin{equation}
H_{t}(\partial B_{u}) \subset\pi_{x} g(t, \partial B_{u} ) \subset
\mathbb{R}^{u} \setminus\overline{B}_{u}.
\end{equation}
Hence
\begin{equation}
0 \notin H_{t}(\partial B_{u}).
\end{equation}
Therefore from the homotopy property of local degree it follows that
\begin{equation}
\deg(H_{t},B_{u},0)=\deg(H_{1},B_{u},0) \quad t \in[0,1].
\end{equation}
For $t=1$ we have
\begin{equation}
H_{1}(x)=Ax,
\end{equation}
and since $A$ is a nonsingular linear map, therefore
\begin{equation}
\deg(H_{t},B_{u},0)=\det{A}=\pm1 \quad t \in[0,1],
\end{equation}
hence $\deg(H_{t},B_{u},0)\ne0$ and by the solution property of
the local degree, there exists an $x$ in $B_{u}$ solving
(\ref{eq:Hhor}). Therefore there exists a solution for
(\ref{eq:x-of-z}).

Observe that this solution is unique. We can take any $x_{0},x_{1}
\in\overline{B}_{u}$, $x_{0} \neq x_{1}$, such that $f(b(x_{0})) \in D$ and
$f(b(x_{1})) \in D$. We will use the notation $q=f(b(x_{0}))$. By Def.
\ref{def:f-cone-cond} we can choose $i(q),j(q)$ such that $(V_{j(q)}%
,U_{i(q)})$ is a cone enclosing pair for $q$ and
\begin{equation}
f(p^{-1}(V_{j(x_{0})})\cap| D| )\subset p^{-1}(U_{i(q)}).
\label{eq:coincide-j(x)1}%
\end{equation}
From the cone condition for $b$ and map $f$ it follows, that
\begin{align}
%\label{eq:con-prop}
Q_{i(q),h}(f_{i(q)i(x_{0})}( \tilde{\eta}_{i(x_{0}%
)}(b(x_{0})) ) - f_{i(q)i(x_{0})}(\tilde{\eta}_{i(x_{0})}(b(x_{1}))))
> \label{eq:cones-for-Vjx} \\
m Q_{i(x_{0}),h}(\tilde{\eta}_{i(x_{0})}(b(x_{0})) - \tilde{\eta}_{i(x_{0}%
)}(b(x_{1}))) > 0.\nonumber
\end{align}

This immediately implies that
\[
\pi_{x} f_{i(q)i(x_{0})}( \tilde{\eta}_{i(x_{0})}(b(x_{0})) \neq\pi_{x}
f_{i(q)i(x_{0})}(\tilde{\eta}_{i(x_{0})}(b(x_{1})).
\]
Therefore, we have a well defined map $v:\overline{B}_{u} \to B_{u}$, given as
the solution of implicit equation $\pi_{x} f_{i(q)i(x_{0})}( \tilde{\eta
}_{i(x_{0})}(b(v(x)))=x$. We define a map $d:\overline{B}_{u} \to
N_{\eta_{i(q)}}(U_{i(q)})$ by $d(x)=f(b(v(x)))$.

Take now any $x_{0}$ such that $q=f(b(x_{0}))\in|D|$, and take a
cone enclosing pair $(V_{j(q)},U_{i(q)})$ for $q$ which satisfies
(\ref{eq:coincide-j(x)1}--\ref{eq:cones-for-Vjx}). We will show
that $d$ is a horizontal disc which satisfies cone condition for
$(V_{j(q)},U_{i(q)})$. From  (\ref{eq:cones-for-Vjx}))  it follows
that for any $x_{1},x_{2} \in\overline{B}_{u}$, $x_{1} \neq x_{2}$
we have
\begin{equation}
Q_{i(q),h}(\eta_{i(q)}(d(x_{1})) - \eta_{i(q)}(d(x_{2})) ) >0. \label{eq:cc-d}%
\end{equation}
Observe that this implies that the map $\tilde{d}=\tilde\eta_{i(q)} \circ d $
is Lipschitz. Namely, for any $x_{1},x_{2}\in\overline{B}_{u}$ from
(\ref{eq:cc-d}) we have
\begin{equation}
\beta_{i(q)}(\pi_{3}(\tilde{d}(x_{1})-\tilde{d}(x_{2}))) + \gamma_{i(q)}%
(\pi_{1}(\tilde{d}(x_{1})-\tilde{d}(x_{2}))) \leq\alpha_{i(q)}(\pi_{2}%
(\tilde{d}(x_{1})-\tilde{d}(x_{2})),
\end{equation}
therefore we obtain
\begin{align*}
M_{-} \| \pi_{3} (\tilde{d}(x_{1})-\tilde{d}(x_{2}))\|^{2}+ M_{-}\| \pi_{1}
(\tilde{d}(x_{1})-\tilde{d}(x_{2})) \|^{2}\\
\leq M_{+} \| \pi_{2}(\tilde{d}(x_{1})-\tilde{d}(x_{2}))\|^{2}= M_{+} \|x_{1}-
x_{2}\|^{2} ,
\end{align*}
which ensures continuity of $d$. To finish the argument
that $d$ is a horizontal disk for $(V_{j(q)},U_{i(q)})$ we need
the homotopy. This homotopy is obtained by applying
Lemmas~\ref{lem:hord-is-graph} and~\ref{lem:graph-is_hord}.
\end{proof}

From Lemma \ref{lem:im-of-hor} we obtain by induction the following lemma.

\begin{lemma}
\label{lem:image-of-horizontal} Let $n\geq1$. Assume that
\[
D\overset{f}{\Longrightarrow}D,
\]
and that $f$ satisfies (forward) cone conditions. Let $b:\overline{B}%
_{u}\rightarrow D$. Assume that for any $x\in B^{u}$ we have a $(V_{j(x)}%
,U_{i(x)})$ which is a cone enclosing pair for $b(x)$, and that $b$ is a
horizontal disc which satisfies cone condition for $(V_{j(x)},U_{i(x)})$.
Then
\[
\beta=\{y:y=f^{n}(b(x))\text{ \textit{and }}f^{i}(b(x))\in D\text{ \textit{for
}}i=1,\ldots,n\text{ \textit{for some }}x\in B^{u}\}
\]
is nonempty and for any $y\in\beta$ there exists a $(V_{j(y)},U_{i(y)})$ which
is a cone enclosing pair for $y$, such that $\beta$ is a horizontal disc which
satisfies cone condition for $(V_{j(y)},U_{i(y)})$.
\end{lemma}

Using Lemma \ref{lem:image-of-horizontal} we will now show that on a horizontal
disc which satisfies cone conditions we have a unique point whose forward
iterations remain inside of the ch-set.

\begin{lemma}
\label{lem:horInvUnique} Assume that
\[
D\overset{f}{\Longrightarrow}D,
\]
and that $f$ satisfies (forward) cone conditions. Let
$b:\overline{B}_{u} \to D$. Assume that for any $x\in B_{u}$ we
have a $(V_{j(x)},U_{i(x)})$ which is a cone enclosing pair for
$b(x)$, and that $b$ is a horizontal disc which satisfies cone
condition for $(V_{j(x)},U_{i(x)})$. Then there exists a unique
point $x^{\ast}\in\overline{B}_{u}$, such that
\begin{equation}
f^{n}(b(x^{\ast})) \in\inte D, \qquad n=0,1,2,\dots. \label{eq:fnInD}%
\end{equation}

\end{lemma}

\begin{proof}
Without any loss of generality we can assume that $\phi=\Id$ and therefore
$D_{\phi}=D=N(\Lambda)$, $b_{\phi}=b$ and $f_{\phi}=f$.

From Lemma~\ref{lem:image-of-horizontal} it follows that for any
$n\in\mathbb{N}$ there exists a point $x_{n}\in\overline{B}_{u}$ such that
\[
f^{k}(b(x_{n})) \in D\quad\mbox{for $k=0,1,\ldots,n.$ }
\]
Because $\overline{B}_{u}$ is compact and $D$ is closed, we can pass to a
convergent subsequence and obtain a point $x^{\ast}\in\overline{B}_{u}$, such
that
\begin{equation}
f^{k}(b(x^{\ast}))\in D\quad\mbox{for $k=0,1,\ldots$}.
\end{equation}
In fact we have
\begin{equation}
f^{k}(b(x^{\ast}))\in\inte D\quad\mbox{for $k=0,1,\ldots$}
\end{equation}
because from the definition of covering relation it follows that
points from $D^{+}$ are not in $f(D)$ and points from $D^{-}$ are
mapped out of $D$.

Let us show that such a point $x^{\ast}$ is unique. Let us assume
that we have two points $x_{0},x_{1}\in\overline{B}_{u}$,
$x_{0}\neq x_{1}$ satisfying (\ref{eq:fnInD}). From
Lemma~\ref{lem:image-of-horizontal} it follows that for any
$n=0,1,2,\dots$ points $q_{l,n}$ given by
\begin{equation}
q_{l,n}:=f^{n}(b(x_{l}))\in\inte D, \quad l=0,1
\end{equation}
belong to $b_{n}$, a horizontal disc satisfying cone conditions for a cone
enclosing pair $(V_{j_{n}}, U_{i_{n}})$ for $q_{0,n}$.

 From cone conditions
for $f$ and $b_{n}$ we have
\begin{align}
Q_{i_{n},h}(\tilde{\eta}_{i_{n}}(q_{0,n}) - \tilde{\eta}_{i_{n}}(q_{1,n}) ) >
mQ_{i_{n-1},h}(\tilde{\eta}_{i_{n-1}}(q_{0,n-1}) - \tilde{\eta}_{i_{n-1}%
}(q_{1,n-1}) ).
\end{align}
Hence
\begin{align}
Q_{i_{n},h}(\tilde{\eta}_{n}(q_{0,n}) - \tilde{\eta}_{n}(q_{1,n}) ) > m^{n}
Q_{0,h}(\tilde{\eta}_{i_{0}}(q_{0,0}) - \tilde{\eta}_{i_{0}}(q_{1,0}) ).
\label{eq:cc-iter}%
\end{align}
From (\ref{eq:cc-iter}) it follows that
\begin{align}
4M_{+} \geq\alpha(\pi_{x}\tilde{\eta}_{i_{n}}(q_{0,n})- \pi_{x} \tilde{\eta
}_{i_{n}}(q_{1,n})) \geq\nonumber\\
Q_{{i_{n}},h}(\tilde{\eta}_{i_{n}}(q_{0,n})- \tilde{\eta}_{i_{n}}(q_{1,n}))
>m^{n}Q_{i_{0},h}(\tilde{\eta}_{i_{0}}(q_{0,0})- \tilde{\eta}_{i_{0}}%
(q_{1,0}))\nonumber\\
=m^{n}\alpha(x_{0}-x_{1}) \geq m^{n}M_{-} \| x_{0}-x_{1}\|^{2}.
\nonumber\label{eq:n-power-bound}%
\end{align}
This for $\|x_{0} - x_{1}\| > 0$, since $m>1$, can not be true for all $n
=0,1,2,\dots$. Hence $x_{0}=x_{1}$. This finishes the proof.
\end{proof}

We will now show that points whose forward iterations remain inside of the ch-set form
 vertical discs which satisfy cone conditions.

\begin{theorem}
\label{thm:InvPos-vcc} Let $n\geq1$. Assume that
\[
D\overset{f}{\Longrightarrow}D,
\]
and that $f$ satisfies cone conditions. Then for any $\lambda_{0}\in\Lambda$
there exists a $b:\overline{B}_{s} \rightarrow D$, such that
\[
p(b(\overline{B}_{s})) = \{\lambda_{0}\},
\]
and for any $y\in\overline{B}_{s}$
\begin{equation}
f^{n}(b(y))\in\inte D \quad\mbox{for $n=0,1,2,\ldots$.}
\end{equation}
In addition, $b$ is a vertical disc which satisfies cone
conditions for any cones chart pair $(V_{j},U_{i})$ such that
$\lambda_{0} \in V_{j}$.

What is more, if $p(q)=\lambda_{0}$ and $f^{n}(q)\in D$ for all $n
\in\mathbb{N}$, then $q \in b(\overline{B}_{s})$.
\end{theorem}

\begin{proof}
Without any loss of generality we can assume that $\phi=\Id$, hence $b_{\phi
}=b$ for all vertical or horizontal discs and $f_{\phi}=f$.

Let us choose $\lambda_{0}\in\Lambda$ and some $q \in p^{-1}(\{ \lambda_{0}
\})\cap|D|$. Let $(V_{j},U_{i})$ be a cones chart pair and a cone enclosing
pair for $q$. Let us fix a $y \in\overline{B}_{s}$ and let $b_{0}:\overline
{B}_{u} \to D$ be a horizontal disc defined by
\begin{equation}
b_{0}(x)=\tilde{\eta}_{i}^{-1}(\lambda_{0},x,y).
\end{equation}
Since $(V_{j},U_{i})$ is a cones chart pair, for any
$x\in\overline {B}_{u}$ we have a $(V_{j(x)},U_{i(x)})$ which is a
cone enclosing pair for $b_{0}(x)$. This means that we can apply
Lemma ~\ref{lem:horInvUnique}. Hence there exists a uniquely
defined $x^{\ast}(\lambda_{0},y) \in\overline{B}_{u}$, such that
\begin{equation}
f^{n}(\tilde{\eta}_{i}^{-1}(\lambda_{0},x^{\ast}(\lambda_{0},y),y) \in\inte D,
\quad n=0,1,2,\dots.
\end{equation}

We define a map $b:\overline{B}_{s} \rightarrow D$
\begin{equation}
b(y)=\tilde{\eta}_{i}(\lambda_{0},x^{\ast}(\lambda_{0},y),y).
\label{eq:b-disc-def}%
\end{equation}
We will show that the map $b$ is a disc satisfying the assertion of our theorem.

We need to show that $b$ is a vertical disc which satisfies cone conditions
for $(V_{j},U_{i})$. We will first show that for any two points $y_{1},y_{2}
\in\overline{B}_{s}$ such that $y_{1}\neq y_{2}$ we have
\begin{equation}
Q_{i,v}\left(  \tilde{\eta}_{i}(b(y_{1})) - \tilde{\eta}_{i}(b(y_{2}))
\right)  <0. \label{eq:lip-on-y}%
\end{equation}
Let us assume that (\ref{eq:lip-on-y}) does not hold. Then since
$\pi _{\lambda}(b(y_{1}))=\pi_{\lambda}(b(y_{2}))=\lambda_0$ we
have
\begin{equation}
Q_{i,h}\left(  \tilde{\eta}_{i}(b(y_{1})) - \tilde{\eta}_{i}(b(y_{2}))\right)
= Q_{i,v}\left(  \tilde{\eta}_{i}(b(y_{1})) - \tilde{\eta}_{i}(b(y_{2}%
))\right)  \geq0. \label{eq:vert-contr}%
\end{equation}
Let $(V_{j_{1}}, U_{i_{1}})$ be a cone enclosing pair for
$f(b(y_{1}))$ such that $f(p^{-1}(V_{j})\cap|D|)\subset
p^{-1}(U_{i_{1}})$. From the fact that $f$ satisfies cone
conditions we have
\begin{equation}
Q_{i_{1},h}(\tilde{\eta}_{i_{1}}(f(b(y_{1}))) - \tilde{\eta}_{i_{1}}%
(f(b(y_{2})))) > m Q_{i,h}\left(  \tilde{\eta}_{i}(b(y_{1})) - \tilde{\eta
}_{i}(b(y_{2}))\right)  \geq0. \label{eq:fb-in-dir}%
\end{equation}
Let
us take any horizontal disc satisfying the cone condition for
$(V_{j_{1}} , U_{i_{1}})$ passing through $f(b(y_{1}))$ and
$f(b(y_{2}))$. From Lemma~\ref{lem:horInvUnique} it follows that
since $f^{n}(b(y_{i}))
\in\inte D$ for $i=1,2$ and all $n=0,1,2,\dots$ that $f(b(y_{1}))=f(b(y_{2}%
))$, but this is in a contradiction with (\ref{eq:fb-in-dir}). Hence
(\ref{eq:lip-on-y}) holds.

From (\ref{eq:lip-on-y}) we have
\begin{align}
0 >Q_{i,v}\left(  \tilde{\eta}_{i}(b(y_{1})) - \tilde{\eta}_{i}(b(y_{2}))
\right)  =\nonumber\\
\alpha_{i}( \pi_{x} (\tilde{\eta}_{i}(b(y_{1}))) - \pi_{x} (\tilde{\eta}%
_{i}(b(y_{2})))) - \beta_{i}(\pi_{y} (\tilde{\eta}_{i}(b(y_{1}))) - \pi_{y}
(\tilde{\eta}_{i}(b(y_{2}))))\nonumber\\
=\alpha(x^{\ast}(\lambda_{0},y_{1})-x^{\ast}(\lambda_{0},y_{2}))-\beta
(y_{1}-y_{2}). \label{eq:vert-condition1}%
\end{align}
From (\ref{eq:vert-condition1}) and (\ref{eq:M-est}) we have
\begin{equation}
\|x^{\ast}(\lambda_{0},y_{1})-x^{\ast}(\lambda_{0},y_{2})\|^{2}\leq\frac{
M_{+}}{M_{-}}\|y_{1}-y_{2}\|^{2},
\end{equation}
which together with  (\ref{eq:b-disc-def})  $b$ guarantees its
continuity. This by Lemma \ref{lem:graph-is_hord} means that $b$
is a vertical disc. Condition (\ref{eq:lip-on-y}) now ensures that
it satisfies cone conditions.
\end{proof}

We now come to our main result. Theorem \ref{th:main} will give us
the existence of a normally hyperbolic invariant manifold together
with its stable an unstable manifolds. At this stage, the
assumptions of the theorem might seem somewhat abstract. In
Section \ref{sec:ver-cond} we will show how these assumptions can
be verified in practice, and in Section \ref{sec:rel-classic} we
will highlight how the theorem relates
 to the classical version of the normally hyperbolic invariant manifold theorem \cite{HPS}.

\begin{theorem}
\label{th:main} Let $(D,$ $\{\eta_{i},U_{i},Q_{i,h},Q_{i,v}\}_{i\in I},$
$\{V_{j}\}_{j\in J})$ be a ch-set with cones. Let $f:\xi_{u}\oplus\xi
_{s}\rightarrow\xi_{u}\oplus\xi$ be a homeomorphism. If $f$ satisfies forward
and backward cone conditions, then:

\begin{enumerate}
\item There exists a continuous function
\[
\chi:\Lambda\rightarrow\mathrm{int}D,
\]
such that $p(\chi(\lambda))=\lambda$ for all $\lambda\in\Lambda,$ and
\[
\chi(\Lambda)=\mathrm{inv}(f,D):=\{z\in D|f^{n}(z)\in D\text{ for all }%
n\in\mathbb{Z}\}.
\]

\item There exist $C^{0}$
submanifolds $W^{u}$ and $W^{s}$ such that
\[
W^{u}\cap W^{s}=\chi(\Lambda),
\]
$W^{u}$ consists of all points whose backward iterations converge to
$\chi(\Lambda),$ and $W^{s}$ consists of all points whose forward iterations
converge to $\chi(\Lambda)$.
\end{enumerate}
\end{theorem}
\begin{remark}
Let us note that the stable and unstable manifolds $W^u$ and $W^s$ from Theorem \ref{th:main} are different from the vector bundles $\xi_u$ and $\xi_s$. The bundles $\xi_u$ and $\xi_s$ provide only approximate coordinates in which we look for $W^u$ and $W^s$, in practice $W^u$ will be close to $\xi_u$ and $W^s$ close to $\xi_s$, but they need not precisely match.
\end{remark}
\begin{proof}[Proof of Theorem \ref{th:main}]
Observe first that vertical cones for the inverse map coincide with horizontal
cones for the forward map (see (\ref{eq:cone-h}), (\ref{eq:cone-v-inv}))
\begin{equation}
C^{-}(Q_{i,v}^{-})=C^{+}(Q_{i,h}). \label{eq:cones-sym}%
\end{equation}
Take $\lambda_{0}$ from $\Lambda.$ From the fact that $f$ satisfies cone
conditions, by Theorem \ref{thm:InvPos-vcc} we have a vertical disc
$b:B_{s}\rightarrow D\cap p^{-1}(\lambda_{0})$ of forward invariant points
($b$ is vertical for the ch-set $(D,$ $\{\eta_{i},U_{i},Q_{i,h},Q_{i,v}%
\}_{i\in I},$ $\{V_{j}\}_{j\in J})$). Since $f^{-1}$ satisfies cone
conditions, by Theorem \ref{thm:InvPos-vcc} we also have a "vertical" disc
$b^{-}:B_{u}\rightarrow D\cap p^{-1}(\lambda_{0})$ of backward invariant
points ($b^{-}$ is vertical with respect to the ch-set $(D,$ $\{\eta_{i}%
,U_{i},Q_{i,h}^{-},Q_{i,v}^{-}\}_{i\in I},$ $\{V_{j}\}_{j\in J})$ with
reversed roles of $u$ and $s$). We can define%
\[
\chi(\lambda_{0}):=b(B_{s})\cap b^{-}(B_{u}).
\]
The function is $\chi$ properly defined since by (\ref{eq:cones-sym}) $b^{-}$
is a horizontal disc for the ch-set $(D,$ $\{\eta_{i},U_{i},Q_{i,h}%
,Q_{i,v}\}_{i\in I},$ $\{V_{j}\}_{j\in J})$. A horizontal disc and a vertical
disc which satisfy cone conditions and are contained in $p^{-1}(\lambda_{0})$
intersect at a single point.

We now have to show that $\chi$ is continuous. Take any sequence $\lambda
_{n}\rightarrow\lambda_{0}.$ By the fact that $D$ is compact we can take a
convergent subsequence $\chi(\lambda_{v_{n}})\rightarrow z.$ We need to show
that $z=\chi(\lambda_{0}).$ For any $k\in\mathbb{N}$, by continuity of
functions $f,$ $f^{-1}$ and closeness of $D$, we know that $\lim
_{n\rightarrow\infty}f^{k}(\lambda_{v_{n}}),\lim_{n\rightarrow\infty}%
f^{-k}(\lambda_{v_{n}})\in D,$ hence $z\in\mathrm{inv}(f,D).$ The fact that
$z=\chi(\lambda_{0})$ follows from the uniqueness of the choice of
$\chi(\lambda_{0}).$

Now we will construct the manifold $W^{s}.$ For any $\lambda$ from $\Lambda$,
by Theorem \ref{thm:InvPos-vcc} we have a vertical disc $b_{\lambda}%
:B_{s}\rightarrow D\cap p^{-1}(\lambda)$ of forward invariant points. We can
define%
\begin{equation}
W^{s}=\bigcup\limits_{\lambda\in\Lambda}b_{\lambda}(B_{s}).\label{eq:Ws}%
\end{equation}
We need to show that $W^{s}$ is $C^{0}.$ We take any Cauchy
sequence $x_{n}$ from $W^{s}.$ From the compactness of $D$ we know
that $x_{n}$ converges to some $x_{0}\in D.$ We will show that
$x_{0}\in W^{s}.$ Since for any $k,n>0$
we have $f^{k}(x_{n})\in D$, by continuity of $f$ we have that $f^{k}%
(x_{0})\in D.$ Letting $\lambda_{0}=\pi_{\lambda}(\phi(x_{0}))$,
we can see that $x_{0}\in p^{-1}(\lambda_{0}),$ which by Theorem
\ref{thm:InvPos-vcc} means that $x_{0}$ lies on the unique
vertical disc $b_{\lambda_{0}}$ of forward invariant points in
$p^{-1}(\lambda_{0}),$ hence by (\ref{eq:Ws}) we have $x_{0}\in
W^{s}.$

The construction of $W^{u}$ is analogous.

We will now show that for any $x_{0}\in W^{s}$ $f^{n}(x_{0})$ converges to  $\chi(\Lambda) $ as $n$ goes to infinity.
Let us consider the limit set of the point $x_{0}$%
\[
\omega(f,x_{0})=\{q|\lim_{k\rightarrow\infty}f^{n_{k}}(x_{0})=q\text{ for some
}n_{k}\rightarrow\infty\}.
\]
If we can show that $\omega(f,x_{0})$ is contained in $W^{u}\cap W^{s}=\chi(\Lambda),$
then this will conclude our proof. We take any $q=\lim_{k\rightarrow\infty
}f^{n_{k}}(x_{0})$ from $\omega(f,x_{0}).$ We need to show that $q\in
W^{u}\cap W^{s}.$ By continuity of $W^{s}$ we know that $q\in W^{s}.$ Suppose
now that $q\notin W^{u}.$ This would mean that there exists an $n>0$ for which
$f^{-n}(q)\notin D.$ Since%
\[
\lim_{k\rightarrow\infty}f^{n_{k}-n}(x_{0})=f^{-n}(q),
\]
we have that $f^{-n}(q)\in\omega(f,x_{0}),$ but this contradicts the fact that
$\omega(f,x_{0})\subset D.$

Showing that all backward iterations of points in $W^u$ converge to $\chi(\Lambda)$ is analogous.
\end{proof}

%TCIDATA{OutputFilter=latex2.dll}
%TCIDATA{Version=5.00.0.2606}
%TCIDATA{LaTeXparent=0,0,online-edit.tex}
\section{Rigorous numerical verification of covering and cone conditions} \label{sec:ver-cond}

In this section we will show how assumptions of Theorem \ref{th:main} can be verified. We set up
our conditions  so that they are verifiable using rigorous numerics. To verify them it is enough
to obtain estimates of the derivatives of the local maps on the ch-set. We will show that
the assumptions of Theorem \ref{th:main} follow from explicit algebraic conditions on these estimates.
An example of how this works in practice will be shown in Section \ref{sec:Henon}.

We start with the definition of an interval enclosure of a derivative.
\begin{definition}
Let $U\subset\mathbb{R}^{n}$ and $f:U\rightarrow\mathbb{R}^{n}$ be a $C^{1} $
function. We define the \emph{interval enclosure} of $df$ on the set $U$ as
\[
\lbrack df(U)]=\left\{  A\in\mathbb{R}^{n\times n}|A_{ij}\in\left[  \inf_{x\in
U}\frac{df_{i}}{dx_{j}}(x),\sup_{x\in U}\frac{df_{i}}{dx_{j}}(x)\right]
\text{ for all }i,j=1,\ldots,n\text{ }\right\}  .
\]

\end{definition}

\subsection{Verification of the covering condition}

In this section we show how covering relations on ch-sets follow from algebraic conditions
on the derivatives of local maps.

\begin{theorem}
\label{th:diff-cover}Assume that $(D,\{\eta_{i},U_{i},Q_{i,h},Q_{i,v}\}_{i\in
I},\{V_{j}\}_{j\in J})$ is a ch-set with cones, with convex $\eta_{i}(V_{j})$
for all $V_{j}\subset U_{i}$. Assume that $f:D\rightarrow\xi_{u}\oplus\xi_{s}$
is such that for any $V_{j}\subset U_{i_{0}}$ there exists $i_{1}$ for which
\[
f(p^{-1}(V_{j})\cap|D|)\subset p^{-1}(U_{i_{1}}).
\]
Assume that for any such $j,i_{0},i_{1}$ the function $f_{i_{1}i_{0}}$ is
differentiable, and for any $\theta\in\eta_{i_{0}}(V_{j})$ we have
\begin{equation}
f_{i_{1}i_{0}}(\theta,0,0)\in\eta_{i_{1}}(U_{i_{1}})\times\overline{B_{u}%
}(0,\varepsilon_{u})\times\overline{B_{s}}(0,\varepsilon_{s}%
)\label{eq:zero-image-lem}%
\end{equation}
for some $1 > \varepsilon_{u},\varepsilon_{s}>0$
($\varepsilon_{u},\varepsilon _{s}$ can be dependent on the choice
of $j,i_{0},i_{1}$). If for any matrix $A\in\lbrack
Df_{i_{1}i_{0}}\left(  N_{\eta_{i_{0}}}(V_{j})\right)  ]$,
the following conditions hold%
\begin{align}
\inf\{\Vert(\pi_{x}\circ A)(0,x,y)\Vert &  :||x||=1,||y||\leq1\}>1+\varepsilon
_{u},\label{eq:cor-exp}\\
\sup\{\Vert(\pi_{y}\circ A)(0,x,y)\Vert &  :||x||\leq1,||y||\leq
1\}<1-\varepsilon_{s},\label{eq:cor-contr}%
\end{align}
then $D\overset{f}{\Longrightarrow}D.$
\end{theorem}

\begin{proof}
For $j,i_{0},i_{1}$ such that
\[
f(p^{-1}(V_{j})\cap|D|)\subset p^{-1}(U_{i_{1}}),
\]
we need to define the homotopy $h=h_{i_{1}i_{0},j}$ from Definition
\ref{def:f-covers}. We will use the notations $p=(x,y)$ and $f_{c}%
:=\pi_{\theta}\circ f_{i_{1}i_{0}}$, $f_{u}:=\pi_{x}\circ f_{i_{1}i_{0}}$,
$f_{s}:=\pi_{y}\circ f_{i_{1}i_{0}}$. We take any $\theta^{\ast}\in\eta
_{i_{0}}(V_{j})$ and define $h$ as%
\begin{multline*}
h_{\alpha}(\theta,p)=\left(  f_{c}((1-\alpha)\theta+\alpha\theta^{\ast
},(1-\alpha)p),(1-\alpha)f_{u}(\theta,0),(1-\alpha)f_{s}(\theta,0)\right)  \\
+\left(  0,\pi_{x}\left(  \int_{0}^{1}
Df_{i_1 i_0}((1-\alpha)\theta+\alpha\theta^{\ast},(1-\alpha)tp)\,dt\cdot\left(0,x,(1-\alpha)y\right)  \right)  \right.  \\
,\left.  \pi_{y}\left(  (1-\alpha)\int_{0}^{1} Df_{i_1 i_0}(\theta,tp)dt\cdot (0,p)\right)  \right)  .
\end{multline*}
Since
\[
f_{i_{1}i_{0}}(\theta,p)-f_{i_{1}i_{0}}(\theta,0)=\int_{0}^{1} Df_{i_1 i_0}(\theta,tp)\,dt\cdot (0,p),
\]
we have $h_{0}(\theta,p)=f_{i_{1}i_{0}}(\theta,p).$ For $\alpha=1$ we have%
\[
h_{1}(\theta,p)=\left(  \lambda^{\ast},Ax,0\right)  ,
\]
with%
\begin{align*}
\lambda^{\ast} &  =f_{c}(\theta^{\ast},0),\\
Ax &  =\pi_{x}\left(  Df_{i_1 i_0}(\theta^{\ast
},0,0)\cdot\left(  0,x,0\right)  \right)  .
\end{align*}

For any $\alpha$ from $[0,1],$ from the fact that
\[
A^{\alpha}:=\int_{0}^{1}Df_{i_1 i_0}%
((1-\alpha)\theta+\alpha\theta^{\ast},(1-\alpha)tp)dt\in\left[  Df_{i_1 i_0}(N_{\eta_{i_{0}}}(V_{j}))\right]  ,
\]
for any $(\theta,p)\in N_{\eta_{i_{0}}}^{-}(V_{j})$ using
(\ref{eq:zero-image-lem}) and (\ref{eq:cor-exp}) we have%
\begin{align}
&  |\pi_{x}(h_{\alpha}(\theta,p))|\label{eq:tmp2}\\
&  =\left\vert (1-\alpha)f_{u}(\theta,0)\phantom{\int_0^1}\right.  \nonumber\\
&  \left.  +\pi_{x}\left(  \int_{0}^{1} Df_{i_1 i_0}((1-\alpha)\theta+\alpha\theta^{\ast},(1-\alpha)tp)\,dt\cdot\left(
0,x,(1-\alpha)y\right)  \right)  \right\vert \nonumber\\
&  \geq\left\vert \pi_{x}\left(  \int_{0}^{1} Df_{i_1 i_0}((1-\alpha)\theta+\alpha\theta^{\ast},(1-\alpha
)tp)dt\cdot\left( 0, x,(1-\alpha)y\right)  \right)  \right\vert -\varepsilon
_{u}\nonumber\\
&  =|\pi_{x}\left(  A^{\alpha}\left(0,  x,(1-\alpha)y\right)  \right)
|-\varepsilon_{u}\nonumber\\
&  >1.\nonumber
\end{align}
This proves that for any $\alpha\in\lbrack0,1]$ we have $h_{\alpha}%
(N_{\eta_{i_{0}}}^{-}(V_{j}))\cap N_{\eta_{i_{1}}}^{-}(U_{i_{1}})=\emptyset.$
Also for $\alpha=1$  since $A(x)=A^{1}(0,x,0)=\pi_{x}%
(h_{1}(\theta,x,0))$ from (\ref{eq:tmp2}) we have $A\left(  \partial B_{u}(0,1)\right)
\cap\overline{B_{u}}(0,1)=\emptyset$.

For $(\theta,p)\in N_{\eta_{i_{0}}}^{-}(V_{j}),$ using
(\ref{eq:zero-image-lem}) and (\ref{eq:cor-contr}) we have%
\begin{align*}
|\pi_{y}\left(  h_{\alpha}(\theta,p)\right)  | &  =\left\vert (1-\alpha
)f_{s}(\theta,0)+\pi_{y}\left(  (1-\alpha)\int_{0}^{1} Df_{i_1 i_0}(\theta,tp)dt\cdot (0, p)\right)  \right\vert \\
&  \leq\left\vert (1-\alpha)\pi_{y}\left(  \int_{0}^{1} Df_{i_1 i_0}(\theta,tp)dt\cdot (0,p) \right)  \right\vert
+\varepsilon_{s}\\
&  =(1-\alpha)|\pi_{y}(A^{0}(0,p))|+\varepsilon_{s}\\
&  <1,
\end{align*}
which means that $h_{\alpha}(N_{\eta_{i_{0}}}(V_{j}))\cap N_{\eta_{i_{1}}}%
^{+}(U_{i_{1}})=\emptyset.$ This finishes our proof.
\end{proof}

\subsection{Verification of cone conditions}

In this section we will show how to verify cone conditions using interval enclosures of derivatives of local maps.

We start with a technical lemma.

\begin{lemma}
\label{lem:A-bounds}Let $p=(p_{1},p_{2},p_{3})\in\mathbb{R}^{c}\times
\mathbb{R}^{u}\times\mathbb{R}^{s}$ and
\begin{equation}
A=\left(
\begin{array}
[c]{lll}%
A_{11} & A_{12} & A_{13}\\
A_{21} & A_{22} & A_{23}\\
A_{31} & A_{32} & A_{33}%
\end{array}
\right)  \label{eq:A-matrix}%
\end{equation}
be an $(c+u+s)\times(c+u+s)$ matrix. If $Q_h(p_1,p_2,p_2)=-\|p_1\|^2+\|p_2\|^2-\|p_3\|^2$, then%
\[
Q_{h}\left(  Ap\right)  \geq-\mathbf{a}\left\Vert p_{1}\right\Vert
^{2}+\mathbf{b}\left\Vert p_{2}\right\Vert ^{2}-\mathbf{c}\left\Vert
p_{3}\right\Vert ^{2},
\]
where%
\begin{align}
\mathbf{a} &  =\left\Vert A_{11}\right\Vert ^{2}-\left\Vert A_{21}\right\Vert
_{m}^{2}+\left\Vert A_{31}\right\Vert ^{2}+\sum_{i=1}^{3}\left\Vert
A_{i1}\right\Vert \left(  \left\Vert A_{i2}\right\Vert +\left\Vert
A_{i3}\right\Vert \right)  ,\nonumber\\
\mathbf{b} &  =-\left\Vert A_{12}\right\Vert ^{2}+\left\Vert A_{22}\right\Vert
_{m}^{2}-\left\Vert A_{32}\right\Vert ^{2}-\sum_{i=1}^{3}\left\Vert
A_{i2}\right\Vert \left(  \left\Vert A_{i1}\right\Vert +\left\Vert
A_{i3}\right\Vert \right)  ,\label{eq:coeff-a-b-c}\\
\mathbf{c} &  =\left\Vert A_{13}\right\Vert ^{2}-\left\Vert A_{23}\right\Vert
_{m}^{2}+\left\Vert A_{33}\right\Vert ^{2}+\sum_{i=1}^{3}\left\Vert
A_{i3}\right\Vert \left(  \left\Vert A_{i1}\right\Vert +\left\Vert
A_{i2}\right\Vert \right)  .\nonumber
\end{align}

\end{lemma}

\begin{proof}
Using the fact that%
\[
\pm2\left\langle A_{ki}p_{i},A_{kj}p_{j}\right\rangle
\geq-||A_{ki} || \cdot ||A_{kj}||\left(
||p_{i}||^{2}+||p_{j}||^{2}\right)  ,
\]
we can compute%
\[
Q(Ap)=
\]%
\begin{align*}
&  =-||\sum_{i=1}^{3}A_{1i}p_{i}||^{2}+||\sum_{i=1}^{3}A_{2i}p_{i}%
||^{2}-||\sum_{i=1}^{3}A_{3i}p_{i}||^{2}\\
&  =-\sum_{i,j=1}^{3}\left\langle A_{1i}p_{i},A_{1j}p_{j}\right\rangle
+\sum_{i,j=1}^{3}\left\langle A_{2i}p_{i},A_{2j}p_{j}\right\rangle
-\sum_{i,j=1}^{3}\left\langle A_{3i}p_{i},A_{3j}p_{j}\right\rangle \\
&  =-\sum_{i=1}^{3}||A_{1i}p_{i}||^{2}+\sum_{i=1}^{3}||A_{2i}p_{i}||^{2}%
-\sum_{i=1}^{3}||A_{3i}p_{i}||^{2}-2\sum_{k=1}^{3}\sum_{i<j}\pm
\left\langle
A_{ki}p_{i},A_{kj}p_{j}\right\rangle \\
&  \geq\left\Vert p_{1}\right\Vert ^{2}(-||A_{11}||^{2}+||A_{21}||_{m}%
^{2}-||A_{31}||^{2})\\
&  \quad+\left\Vert p_{2}\right\Vert ^{2}(-||A_{12}||^{2}+||A_{22}||_{m}%
^{2}-||A_{32}||^{2})\\
&  \quad+\left\Vert p_{3}\right\Vert ^{2}(-||A_{13}||^{2}+||A_{23}||_{m}%
^{2}-||A_{33}||^{2})\\
&  \quad-\sum_{k=1}^{3}\sum_{i<j}||A_{ki}||||A_{kj}||\left(  ||p_{i}%
||^{2}+||p_{j}||^{2}\right)  \\
&  =-\mathbf{a}\left\Vert p_{1}\right\Vert ^{2}+\mathbf{b}\left\Vert
p_{2}\right\Vert ^{2}-\mathbf{c}\left\Vert p_{3}\right\Vert ^{2}.
\end{align*}

\end{proof}

The theorem below gives conditions which imply cone conditions.

\begin{theorem} \label{th:ver-conec}
\label{th:cone-ineq}Assume that $(D,\{\eta_{i},U_{i},Q_{i,h},Q_{i,v}\}_{i \in
I},\{V_{j}\}_{j \in J})$ is a ch-set with cones, with convex $\eta_{i}(V_{j})$
for all $V_{j}\subset U_{i}$. Assume that the maps used for $Q_{i,h},Q_{i,v}$
(see (\ref{eq:cone-h}, \ref{eq:cone-v})) are
\begin{equation}
\alpha(x)=\| x\|^{2}, \quad\beta(y)= \| y\|^{2}, \quad\gamma(\theta) = \|
\theta\|^{2}.
\end{equation}
Assume that for any $x\in D$ and any $(V_{j_{0}},U_{i_{0}})$ which is a cone enclosing
pair for $x$, if $f(x)\in D$ then there exist $i_{1}\in I$, $j_{1}\in J$ such
that $(V_{j_{1}}, U_{i_{1}})$ is a cone enclosing pair for $f(x)$ and $f(p^{-1}%
(V_{j_{0}})\cap| D| ) \subset p^{-1}(U_{i_{1}})$. What is more for any
\begin{align*}
A &  \in\lbrack df_{i_{1} i_{0}}( N_{\eta_{i_{0}}} (V_{j})) ],\\
A &  =\left(  A_{ik}\right)  _{i,k=1,...,3},
\end{align*}
we assume that we have%
\begin{equation}%
\begin{array}
[c]{lll}%
||A_{11}||\leq C, & ||A_{12}||\leq\varepsilon_{c}, & ||A_{13}||\leq
\varepsilon_{c},\\
\mu\leq||A_{21}||_{m}\leq||A_{21}||\leq M, & \alpha\leq||A_{22}||_{m}%
\leq||A_{22}||\leq\mathcal{A}, & ||A_{23}||\leq\varepsilon_{u},\\
||A_{31}||\leq M, & ||A_{32}||\leq\varepsilon_{s}, & ||A_{33}||\leq\beta
\end{array}
\label{eq:A-bounds}%
\end{equation}
(The $C,M,\mu,\mathcal{A},\alpha,\beta, \varepsilon_{u}, \varepsilon_{s} $ and
$\varepsilon_{c} $ can depend on $j_{0}$, $i_{0}$, $j_{1}$ and $i_{1}$). If
there exists an $m>1$ such that%
\begin{align}
C^{2}-\mu^{2}+M^{2}+2C\varepsilon_{c}+M\left(  \mathcal{A}+\varepsilon
_{u}+\varepsilon_{s}+\beta\right)   &  <m,\nonumber\\
-\varepsilon_{c}^{2}+\alpha^{2}-\varepsilon_{s}^{2}-\varepsilon_{c}\left(
C+\varepsilon_{c}\right)  -\mathcal{A}\left(  M+\varepsilon_{u}\right)
-\varepsilon_{s}\left(  M+\beta\right)   &  >m,\label{eq:cone-est1}\\
\varepsilon_{c}^{2}+\beta^{2}+\varepsilon_{c}\left(  C+\varepsilon_{c}\right)
+\varepsilon_{u}\left(  M+\mathcal{A}\right)  +\beta\left(  M+\varepsilon
_{s}\right)   &  <m,\nonumber
\end{align}
then $f$ satisfies cone conditions.
\end{theorem}

\begin{proof}
We only need to prove condition (\ref{eq:cone-cond-Q}). Take any $x_{1}, x_{2}
\in N_{\eta_{i_{0}}} (V_{j})$ for which $Q_{i_{0},h}(x_{1}-x_{2})\geq 0$. We have%
\[
f_{i_{1} i_{0}}(x_{1})-f_{i_{1} i_{0}}(x_{2})=\int_{0}^{1}df_{i_{1} i_{0}%
}(x_{2}+t(x_{1}-x_{2}))dt\cdot(x_{1}-x_{2}).
\]
Let us define a matrix $A=\int_{0}^{1}df_{i_{1} i_{0}}(x_{2}+t(x_{1}%
-x_{2}))dt\cdot(x_{1}-x_{2}).$ Clearly $A\in\lbrack df_{i_{1} i_{0}}(
N_{\eta_{i_{0}}} (V_{j}))].$ By Lemma \ref{lem:A-bounds} we have%
\begin{align}
Q_{i_1,h}(f_{i_{1} i_{0}}(x_{1})-f_{i_{1} i_{0}}(x_{2})) = &  Q_{i_1,h}(A(x_{1}%
-x_{2}))\label{eq:temp6}\\
\geq & -\mathbf{a}||\pi_{1}(x_{1}-x_{2})||^{2}+\mathbf{b}||\pi_{2}(x_{1}%
-x_{2})||^{2}\nonumber\\
& -\mathbf{c}|\pi_{3}(x_{1}-x_{2})||^{2},\nonumber
\end{align}
with $\mathbf{a},\mathbf{b}$ and $\mathbf{c}$ defined as in
(\ref{eq:coeff-a-b-c}). By using the estimates (\ref{eq:A-bounds}) in the
formulas (\ref{eq:coeff-a-b-c}) for the coefficients $\mathbf{a},\mathbf{b}$
and $\mathbf{c}$, and applying (\ref{eq:cone-est1}), from (\ref{eq:temp6}) we
obtain%
\[
Q_{i_1,h}(f_{i_{1} i_{0}}(x_{1})-f_{i_{1} i_{0}}(x_{2}))>mQ_{i_0,h}(x_{1}-x_{2}).
\]
This finishes our proof.
\end{proof}

\begin{remark}
\label{rem:rescaling}It might turn out that the conditions (\ref{eq:cone-est1}%
) will not hold due to the fact that the bounds%
\begin{align*}
\left\|\frac{d(\pi_{x}f_{i_{1}i_{0}})}{d\theta}(N_{\eta_{i_{0}}}(V_{j}))\right\|
&  \leq
M,\\
\left\|\frac{d(\pi_{y}f_{i_{1}i_{0}})}{d\theta}(N_{\eta_{i_{0}}}(V_{j}))\right\|
&  \leq M,
\end{align*}
will produce a large coefficient $M.$ In such case, a change of local
coordinates%
\[
(\theta,x,y)\rightarrow(v\theta,x,y),
\]
will result in conditions
\begin{align}
C^{2}-\left(  \frac{\mu}{v}\right)  ^{2}+\left(  \frac{M}{v}\right)
^{2}+2Cv\varepsilon_{c}+\frac{M}{v}\left(  \mathcal{A}+\varepsilon
_{u}+\varepsilon_{s}+\beta\right)   &  <m,\nonumber\\
-\left(  v\varepsilon_{c}\right)  ^{2}+\alpha^{2}-\varepsilon_{s}%
^{2}-v\varepsilon_{c}\left(  C+v\varepsilon_{c}\right)  -\mathcal{A}\left(
\frac{M}{v}+\varepsilon_{u}\right)  -\varepsilon_{s}\left(  \frac{M}{v}%
+\beta\right)   &  >m,\label{eq:cone-est2}\\
\left(  v\varepsilon_{c}\right)  ^{2}+\beta^{2}+v\varepsilon_{c}\left(
C+v\varepsilon_{c}\right)  +\varepsilon_{u}\left(  \frac{M}{v}+\mathcal{A}%
\right)  +\beta\left(  \frac{M}{v}+\varepsilon_{s}\right)   &  <m,\nonumber\\
1  &  <m.\nonumber
\end{align}
The conditions (\ref{eq:cone-est2}) with appropriately large $v$
will hold more readily than (\ref{eq:cone-est1}), since in
practice we usually have the bound $\varepsilon_{c}$ small in
comparison with $M.$
\end{remark}

\begin{remark} \label{rem:imply-cone}
Let us note that if our choice of the local coordinates in the stable and
unstable direction results in having $\varepsilon_{u}=\varepsilon
_{s}=\varepsilon_{c}=0$, then the condition%
\[
\beta\leq C<\alpha,
\]
implies conditions (\ref{eq:cone-est2}) for any $m\in(\max\{C^{2},1\},\alpha^{2})$ and
sufficiently large $v.$
\end{remark}

%TCIDATA{OutputFilter=latex2.dll}
%TCIDATA{Version=5.00.0.2606}
%TCIDATA{LaTeXparent=0,0,online-edit.tex}

\section{Comparison of the results with the classical normally hyperbolic invariant manifold theorem} \label{sec:rel-classic}

In the classical version of the normally hyperbolic invariant
manifold theorem (see \cite{HPS}) we have the following setting.
We have a smooth Riemann manifold $M$ and a diffeomorphism $f:M
\to M$ with an invariant submanifold $\Lambda \subset M$
\[ f(\Lambda)=\Lambda.\]
We say that $f$ is \emph{$r$-normally hyperbolic} at $\Lambda$ if the tangent bundle of $M$ splits into invariant
by the tangent of $f$ subbundles
\[ T_V M = \xi^u \oplus T \Lambda \oplus \xi^s, \]
such that
\begin{enumerate}
\item $Tf$ expands $\xi^u$ more sharply than $Tf^r$ expands $T \Lambda$,
\item $Tf$ contracts $\xi^s$ more sharply than $Tf^r$ contracts $T\Lambda$.
\end{enumerate}
\begin{theorem}\cite{HPS} \label{th:HPS}
Let $f$ be $r$-normally hyperbolic at $\Lambda$. Through $\Lambda$ pass stable and unstable manifolds
 invariant by $f$ and tangent at $\Lambda$ to $T\Lambda \oplus \xi^s$, $\xi^u \oplus T\Lambda$.
 They are of class $C^r$. The stable manifold is invariantly fibred  by $C^r$ submanifolds tangent
 at $\Lambda$ to the subspaces $\xi^s$. Similarly for the unstable manifold and $\xi^u$. These structures
 are unique, and permanent under small perturbations of $f$.
\end{theorem}

We will now highlight how our result obtained in
Theorem \ref{th:main} (see also Theorems \ref{th:diff-cover}, \ref{th:ver-conec} and Remark \ref{rem:imply-cone})
relates to Theorem \ref{th:HPS}.

In our approach we do not need to start with a normally hyperbolic invariant manifold. We start with a region
(ch-set) $D$ in which we suspect the manifold to be contained. Theorem \ref{th:main} gives us both
the existence an invariant manifold, and of the stable and unstable manifolds inside of $D$.

Using Theorem \ref{th:HPS} it is not straightforward to obtain an
estimate on the size of the perturbation under which the
structures persist. This is due to the fact that the proof is
conducted by the use of the implicit function theorem on infinite
dimensional functional spaces. In contrast, our proof is designed
so that this estimate is explicitly given (see Theorems
\ref{th:diff-cover}, \ref{th:ver-conec}, and Section
\ref{sec:Henon} for an example how this is done in the case of the
rotating H\'{e}non map). This was possible due to the fact that
the proof relies only on topological arguments, and is conducted
in the phase space of the system.

Theorem \ref{th:HPS} is in many ways stronger than our result. Due to the fact that our proof relies
 only on topological arguments we have lost the $C^r$ regularity results and our manifolds are only $C^0$.
  As of yet the method also does not give us the fibration of the stable and unstable manifolds. This issue
   will be addressed in forthcoming work. We also point out that in the method of verification of cone conditions
   given by Theorem \ref{th:ver-conec} (see also Remark \ref{rem:imply-cone}) we assume that we have
    strong expansion properties for the first iterate of the map. This is not required for classical normal
    hyperbolicity, where it is enough that one has strong expansion and contraction properties for some higher
    iteration of the map (for more details see \cite{HPS}). This means that we can apply our results in the
    classical setting, but to do so we need to consider a higher iterate of the map, and take the higher
    iterate to be the function considered in Theorem \ref{th:main}.

	%TCIDATA{OutputFilter=latex2.dll}
	%TCIDATA{Version=5.00.0.2606}
	%TCIDATA{LaTeXparent=0,0,online-edit.tex}
	\section{Rotating H\'{e}non map} \label{sec:Henon}

	In this Section we will apply Theorem \ref{th:main} to obtain an invariant
	manifold for the rotating H\'{e}non map. We will obtain an explicit estimate of the region in which
	the manifold is contained. The size of this region depends on the size of the perturbation. The smaller
	the perturbation is the more exact our estimate.

	\subsection{Statement of the problem.}

	We will consider the rotating H\'{e}non map%
	\begin{align}
	\bar{\theta}  &  =\theta+\omega\quad\text{(mod }1\text{),}\nonumber\\
	\bar{x}  &  =1+y-ax^{2}+\varepsilon\cos(2\pi\theta), \label{eq:rotating-henon}%
	\\
	\bar{y}  &  =bx.\nonumber
	\end{align}
	The dynamics of (\ref{eq:rotating-henon}) with $a=0.68$ and
	$b=0.1$ has been investigated by Haro and de la Llave in
	\cite{HL}, for a demonstration of a numerical algorithm for
	finding invariant manifolds and their whiskers in quasi
	periodically forced systems.

	In this section we will prove that for the the same parameters $a$ and $b$ for
	all $\varepsilon\leq\frac{1}{2},$ there exists an invariant $C^{0}$ manifold
	of (\ref{eq:rotating-henon}) which is homeomorphic to $\mathbb{T}^{1}$ and is
	contained in a set%
	\begin{equation}
	U_{\varepsilon}=\mathbb{T}^{1}\times\lbrack x_{-}-1.1\varepsilon
	,x_{-}+1.1\varepsilon]\times\lbrack y_{-}-0.12\varepsilon,y_{-}%
	+0.12\varepsilon], \label{eq:U-epsilon}%
	\end{equation}
	where $(x_{-},y_{-})$ is a fixed point for the (standard) H\'{e}non map,%
	\begin{align*}
	x_{-}  &  =\frac{-(1-b)-\sqrt{(1-b)^{2}+4a}}{2a}\approx-2.043\,3,\\
	y_{-}  &  =bx_{-}\approx-0.204\,33.
	\end{align*}

	\subsection{The unperturbed map.}

	We start by investigating the case of $\varepsilon=0.$ We will ignore the
	coordinate $\theta$ and concentrate on a map%
	\[
	F(x,y)=(1+y-ax^{2},bx).
	\]
	The point $(x_{-},y_{-})$ is one of the two fixed points $(x_{\pm},y_{\pm})$
	of the map $F$%
	\[
	x_{\pm}=\frac{-(1-b)\pm\sqrt{(1-b)^{2}+4a}}{2a},\quad y_{\pm}=bx_{\pm}.
	\]
	We have%
	\[
	DF(x,y)=\left(
	\begin{array}
	[c]{cc}%
	-2ax & 1\\
	b & 0
	\end{array}
	\right)  ,
	\]
	with two eigenvalues $\lambda_{1}=-ax+\sqrt{b+a^{2}x^{2}},$ $\lambda
	_{2}=-ax-\sqrt{b+a^{2}x^{2}}.$ For $(x_{-},y_{-})$ the eigenvalues are%
	\[
	\lambda_{1}\approx2.\,\allowbreak814\,4,\quad\lambda_{2}\approx
	-3.\,\allowbreak553\,1\times10^{-2}.
	\]
	We will consider the following Jordan forms of the matrix $DF(x_{-},y_{-})$%
	\begin{gather*}
	DF(x_{-},y_{-})=\Phi_{\varepsilon}^{-1}J\Phi_{\varepsilon},\\
	\Phi_{\varepsilon}^{-1}=\varepsilon\kappa\left(
	\begin{array}
	[c]{cc}%
	\tau & \eta\\
	-\tau\lambda_{2} & -\lambda_{1}\eta
	\end{array}
	\right)  ,\quad J=\left(
	\begin{array}
	[c]{cc}%
	\lambda_{1} & 0\\
	0 & \lambda_{2}%
	\end{array}
	\right)  ,\quad\Phi_{\varepsilon}=\frac{1}{\varepsilon}\left(
	\begin{array}
	[c]{cc}%
	-\frac{1}{\tau}\lambda_{1} & -\frac{1}{\tau}\\
	\frac{1}{\eta}\lambda_{2} & \frac{1}{\eta}%
	\end{array}
	\right)  ,
	\end{gather*}
	where $\kappa=1/(\lambda_{2}-\lambda_{1}).$ The constants $\tau,\eta$ serve
	the purpose of an appropriate rescaling of the stable and unstable directions
	in the local coordinates, and will be chosen later on. When we will consider
	the perturbed H\'{e}non map in Section \ref{sec:cover-ver}, for a given
	$\varepsilon>0$ we will use the maps $\Phi_{\varepsilon}^{-1}$ and
	$\Phi_{\varepsilon}.$

	We introduce local coordinates of hyperbolic expansion and
	contraction around the point $(x_{-},y_{-})$ as
	\begin{equation}
	\left(  \tilde{x},\tilde{y}\right)  =\Phi_{\varepsilon}\left(  x-x_{-}%
	,y-y_{-}\right)  . \label{eq:local-coord-henon}%
	\end{equation}
	The map $F$ in the local coordinates is%
	\[
	\tilde{F}(\tilde{x},\tilde{y})=\Phi_{\varepsilon}\left(  F\left(
	\Phi_{\varepsilon}^{-1}(\tilde{x},\tilde{y})+\left(  x_{-},y_{-}\right)
	\right)  -\left(  x_{-},y_{-}\right)  \right)  ,
	\]
	and its derivative $d\tilde{F}$ is equal to%
	\begin{align}
	d\tilde{F}(\tilde{x},\tilde{y})  &  =\Phi_{\varepsilon}\circ dF(\Phi
	_{\varepsilon}^{-1}(\tilde{x},\tilde{y})+\left(  x_{-},y_{-}\right)
	)\circ\Phi_{\varepsilon}^{-1}\nonumber\\
	&  =\Phi_{\varepsilon}\circ dF\left(
	\begin{array}
	[c]{c}%
	\varepsilon\kappa\left(  \tau\tilde{x}+\eta\tilde{y}\right)  +x_{-}\\
	-\varepsilon\kappa\left(  \tau\lambda_{2}\tilde{x}+\eta\lambda_{1}\tilde
	{y}\right)  +y_{-}%
	\end{array}
	\right)  \circ\Phi_{\varepsilon}^{-1}\nonumber\\
	&  =\Phi_{\varepsilon}\circ\left(
	\begin{array}
	[c]{cc}%
	-2a\left(  \varepsilon\kappa\left(  \tau\tilde{x}+\eta\tilde{y}\right)
	+x_{-}\right)  & 1\\
	b & 0
	\end{array}
	\right)  \circ\Phi_{\varepsilon}^{-1}\nonumber\\
	&  =\Phi_{\varepsilon}\circ\left(  \left(
	\begin{array}
	[c]{cc}%
	-2ax_{-} & 1\\
	b & 0
	\end{array}
	\right)  +\left(
	\begin{array}
	[c]{cc}%
	-2a\varepsilon\kappa\left(  \tau\tilde{x}+\eta\tilde{y}\right)  & 0\\
	0 & 0
	\end{array}
	\right)  \right)  \circ\Phi_{\varepsilon}^{-1}\nonumber\\
	&  =J+R_{\varepsilon}, \label{eq:dF-forw}%
	\end{align}
	where%
	\[
	R_{\varepsilon}=-2a\varepsilon\kappa^{2}\left(  \tau\tilde{x}+\eta\tilde
	{y}\right)  \left(
	\begin{array}
	[c]{cc}%
	-\lambda_{1} & -\frac{\eta}{\tau}\lambda_{1}\\
	\frac{\tau}{\eta}\lambda_{2} & \lambda_{2}%
	\end{array}
	\right)  .
	\]
	For any $\tilde{x},\tilde{y}\in\lbrack-1,1]$ we have the following estimates,
	which will be used later on for the verification of the covering and cone
	conditions
	\begin{align}
	&  [d\tilde{F}(B(0,1)\times B(0,1))]\label{eq:der-encl-forw-n}\\
	&  \subset\left(
	\begin{array}
	[c]{cc}%
	\lambda_{1} & 0\\
	0 & \lambda_{2}%
	\end{array}
	\right)  +\varepsilon\left(  \tau+\eta\right)  \left(
	\begin{array}
	[c]{cc}%
	\lbrack-\frac{1}{2},\frac{1}{2}] & [-\frac{1}{2}\frac{\eta}{\tau},\frac{1}%
	{2}\frac{\eta}{\tau}]\\
	\lbrack-\frac{6}{1000}\frac{\tau}{\eta},\frac{6}{1000}\frac{\tau}{\eta}] &
	[-\frac{6}{1000},\frac{6}{1000}]
	\end{array}
	\right)  .\nonumber
	\end{align}

	Now we turn to the inverse map. The inverse map to $F$ is%

	\[
	F^{-1}(x,y)=\left(  \frac{1}{b}y,-1+x+\frac{a}{b^{2}}y^{2}\right)  ,
	\]
	and has a derivative%
	\[
	dF^{-1}(x,y)=\left(
	\begin{array}
	[c]{cc}%
	0 & \frac{1}{b}\\
	1 & \frac{2a}{b^{2}}y
	\end{array}
	\right)  .
	\]
	In the local coordinates (\ref{eq:local-coord-henon}) the inverse map is%

	\[
	\tilde{F}^{-1}(\tilde{x},\tilde{y})=\Phi_{\varepsilon}\left(  F^{-1}\left(
	\Phi_{\varepsilon}^{-1}(\tilde{x},\tilde{y})+\left(  x_{-},y_{-}\right)
	\right)  -\left(  x_{-},y_{-}\right)  \right)  ,
	\]
	and its derivative $d\tilde{F}^{-1}$ is equal to%
	\begin{align}
	d\tilde{F}^{-1}(\tilde{x},\tilde{y})  &  =\Phi_{\varepsilon}\circ dF^{-1}%
	(\Phi_{\varepsilon}^{-1}(\tilde{x},\tilde{y})+\left(  x_{-},y_{-}\right)
	)\circ\Phi_{\varepsilon}^{-1}\nonumber\\
	&  =\Phi_{\varepsilon}\circ dF^{-1}\left(
	\begin{array}
	[c]{c}%
	\varepsilon\kappa\left(  \tau\tilde{x}+\eta\tilde{y}\right)  +x_{-}\\
	-\varepsilon\kappa\left(  \tau\lambda_{2}\tilde{x}+\eta\lambda_{1}\tilde
	{y}\right)  +y_{-}%
	\end{array}
	\right)  \circ\Phi_{\varepsilon}^{-1}\nonumber\\
	&  =\Phi_{\varepsilon}\circ\left(
	\begin{array}
	[c]{cc}%
	0 & \frac{1}{b}\\
	1 & \frac{2a}{b^{2}}\left(  -\varepsilon\kappa\left(  \tau\lambda_{2}\tilde
	{x}+\eta\lambda_{1}\tilde{y}\right)  +y_{-}\right)
	\end{array}
	\right)  \circ\Phi_{\varepsilon}^{-1}\nonumber\\
	&  =\Phi_{\varepsilon}\circ\left(  \left(
	\begin{array}
	[c]{cc}%
	0 & \frac{1}{b}\\
	1 & \frac{2a}{b^{2}}y_{-}%
	\end{array}
	\right)  +\left(
	\begin{array}
	[c]{cc}%
	0 & 0\\
	0 & -\frac{2a}{b^{2}}\varepsilon\kappa\left(  \tau\lambda_{2}\tilde{x}%
	+\eta\lambda_{1}\tilde{y}\right)
	\end{array}
	\right)  \right)  \circ\Phi_{\varepsilon}^{-1}\nonumber\\
	&  =J^{-1}+R_{\varepsilon}^{\prime}, \label{eq:dF-inv}%
	\end{align}
	where%
	\[
	R_{\varepsilon}^{\prime}=\frac{2a}{b^{2}}\varepsilon\kappa^{2}\left(
	\tau\lambda_{2}\tilde{x}+\eta\lambda_{1}\tilde{y}\right)  \left(
	\begin{array}
	[c]{cc}%
	-\lambda_{2} & -\frac{\eta}{\tau}\lambda_{1}\\
	\frac{\tau}{\eta}\lambda_{2} & \lambda_{1}%
	\end{array}
	\right)  .
	\]
	For $\tilde{x},\tilde{y}\in\lbrack-1,1]$ this gives us the following
	estimates
	\begin{align}
	&  [d\tilde{F}^{-1}(B(0,1)\times B(0,1))]\label{eq:der-encl-back}\\
	&  \subset\left(
	\begin{array}
	[c]{cc}%
	\frac{1}{\lambda_{1}} & 0\\
	0 & \frac{1}{\lambda_{2}}%
	\end{array}
	\right)  +\varepsilon\left(  \tau\left\vert \lambda_{2}\right\vert
	+\eta|\lambda_{1}|\right)  \left(
	\begin{array}
	[c]{cc}%
	\lbrack-\frac{6}{10},\frac{6}{10}] & [-50\frac{\eta}{\tau},50\frac{\eta}{\tau
	}]\\
	\lbrack-\frac{\tau}{\eta}\frac{6}{10},\frac{\tau}{\eta}\frac{6}{10}] &
	[-50,50]
	\end{array}
	\right)  .\nonumber
	\end{align}

	\subsection{The atlas\label{sec:atlas}}

	Let us start by introducing a notation $F_{\varepsilon}:\mathbb{T}^{1}%
	\times\mathbb{R}^{2}\rightarrow\mathbb{T}^{1}\times\mathbb{R}^{2}$ for the
	perturbed Henon map (\ref{eq:rotating-henon})%
	\[
	F_{\varepsilon}\left(  \theta,x,y\right)  =\left(  \theta+\omega
	,1+y-ax^{2}+\varepsilon\cos(2\pi\theta),bx\right)  .
	\]

	For computational reasons it will be convenient for us to choose
	$\Lambda=\mathbb{R}/v$ for some large (later specified) number
	$v\in \mathbb{N},$ $v\geq9.$ The $v$ will play the role of the
	rescaling parameter from Remark \ref{rem:rescaling}. We choose
	$p:\mathbb{T}^{1}\times \mathbb{R}^{2}\rightarrow\mathbb{R}/v$ as
	$p(\theta,x,y):=v\theta.$ For $i,j\in\{0,1,\ldots,v\}$ we define
	$V_{j},U_{i}\subset\Lambda$ as
	\begin{align}
	V_{j}  &  :=j+(0,5)\quad\text{mod }v,\label{eq:atlas-henon}\\
	U_{i}  &  :=i+(0,9)\quad\text{mod }v,\nonumber
	\end{align}
	and define maps $\eta_{i}:U_{i}\rightarrow\mathbb{R}$ as
	\begin{equation}
	\eta_{i}(i+x\text{ mod }v):=i+x.
	\label{eq:maps-henon-atl}%
	\end{equation}
	For all $i$ we define  quadratic forms $Q_{i,h},$ $Q_{i,\nu}$ as%
	\begin{align*}
	Q_{i,h}(\theta,x,y)  &  :=x^{2}-y^{2}-\theta^{2},\\
	Q_{i,\nu}(\theta,x,y)  &  :=x^{2}-y^{2}+\theta^{2}.
	\end{align*}
	For any $q=(\theta,x,y)\in N(\Lambda)=\Lambda\times B(0,1)\times B(0,1),$ for
	which $\theta\in\lbrack i,(i+1)]\text{ mod }v$, the pair $(V_{i-2},U_{i-4})$ is both a cones chart pair, as well as a cone enclosing
	pair for $q$ (see Example \ref{ex:sets}).

	We define $\phi:\mathbb{T}^{1}\times\mathbb{R}^{2}\rightarrow\Lambda
	\times\mathbb{R}^{2}$ as%
	\[
	\phi(\theta,x,y):=(v\theta,\Phi_{\varepsilon}\left(  x-x_{-},y-y_{-}\right)
	).
	\]
	Clearly%
	\[
	\phi^{-1}(\theta,x,y)=(\frac{1}{v}\theta,\Phi_{\varepsilon}^{-1}\left(
	x,y\right)  +(x_{-},y_{-})).
	\]
	Finally we define our set $|D|\subset\mathbb{T}^{1}\times\mathbb{R}^{2}$ as%
	\[
	|D|:=\phi^{-1}(N(\Lambda)).
	\]
	Let us note that for different $\varepsilon$ we will have different sets $|D|$.

	With
	the above notations we can see that $(D,\{\eta_{i},U_{i},V_{i},Q_{i,h}%
	,Q_{i,\nu}\})$ is a ch-set with cones.

	\subsection{Verification of the covering conditions.\label{sec:cover-ver}}

	We will show that
	\begin{align}
	&  D\overset{F_{\varepsilon}}{\Longrightarrow}D,\label{eq:forw-cover-hen}\\
	&  D\overset{F_{\varepsilon}^{-1}}{\Longrightarrow}D,
	\label{eq:back-cover-hen}%
	\end{align}
	(with the roles of the stable and unstable coordinates reversed for the inverse map).

	We will first apply Theorem \ref{th:diff-cover} to establish
	(\ref{eq:forw-cover-hen}). From the fact that $F_{\varepsilon}$ is a rotation
	on $\mathbb{T}^{1}$ and from the definition of the sets $V_{j},U_{i}$ (see
	(\ref{eq:atlas-henon})), for any $V_{j}\subset U_{i_{0}}$ there
	exists $i_{1}$ such that $F_{\varepsilon}(p^{-1}(V_{j})\cap|D|)\subset
	p^{-1}(U_{i_{1}}).$
	We will choose the parameters $\varepsilon_{u},\varepsilon_{s}$ for the
	estimates (\ref{eq:cor-exp}) and (\ref{eq:cor-contr}) to be independent from
	the choice of $j,i_{0},i_{1}.$ To simplify notations we will obtain these
	estimates using the map $\left(  F_{\varepsilon}\right)  _{\phi}$ rather than $\left(  F_{\varepsilon}\right)  _{i_{1}i_{0}}.$ We can do so since
	$\eta_{i_{0}},\eta_{i_{1}}$ are identity maps (see (\ref{eq:maps-henon-atl})).

	From the fact that $F_{0}(\theta,x_{-},y_{-})=\left(  \theta+\omega
	,x_{-},y_{-}\right)  ,$ for any $\theta\in \Lambda$ we have%
	\begin{align*}
	\left(  F_{\varepsilon}\right)  _{\phi}(\theta,0,0)  &  =\phi\circ
	F_{\varepsilon}\circ\phi^{-1}\left(  \theta,0,0\right) \\
	&  =\phi\circ F_{\varepsilon}\left(  \frac{1}{v}\theta,x_{-},y_{-}\right) \\
	&  =\phi\left(  F_{0}\left(  \frac{1}{v}\theta,x_{-},y_{-}\right)  +\left(
	0,\varepsilon\cos2\pi\frac{1}{v}\theta,0\right)  \right) \\
	&  =\left(  \theta+v\omega,\Phi_{\varepsilon}\left(  (0,0)+\left(
	\varepsilon\cos2\pi\frac{1}{v}\theta,0\right)  \right)  \right) \\
	&  =\left(  \theta+v\omega,-\frac{1}{\tau}\lambda_{1}\cos2\pi\frac{1}{v}%
	\theta,\frac{1}{\eta}\lambda_{2}\cos2\pi\frac{1}{v}\theta\right)  ,
	\end{align*}
	which gives us%
	\begin{equation}
	\left(  F_{\varepsilon}\right)  _{\phi}(\Lambda,0,0)\subset\Lambda
	\times\overline{B_{u}}\left(  0,\frac{1}{\tau}\left\vert \lambda
	_{1}\right\vert \right)  \times\overline{B_{s}}\left(  0,\frac{1}{\eta
	}\left\vert \lambda_{2}\right\vert \right)  . \label{eq:zero-image}%
	\end{equation}
	Let $g(\theta,x,y)=\left(  0,\cos(2\pi\theta),0\right)  ,$ we then have
	$F_{\varepsilon}=F_{0}+\varepsilon g$ and
	\begin{align}
	d\left(  F_{\varepsilon}\right)  _{\phi}  &  =d\left(  F_{0}\right)  _{\phi
	}+\varepsilon\text{diag}(v\text{id},\Phi_{\varepsilon})dg(\phi^{-1}%
	(\cdot))\text{diag}\left(  \frac{1}{v}\text{id},\Phi_{\varepsilon}^{-1}\right)
	\label{eq:diff-F-phi}\\
	&  =\left(
	\begin{array}
	[c]{cc}%
	1 & 0\\
	0 & d\tilde{F}%
	\end{array}
	\right)  +\left(
	\begin{array}
	[c]{ccc}%
	0 & 0 & 0\\
	\varepsilon\frac{2\pi\kappa}{v\tau}\lambda_{1}\sin2\pi\frac{1}{v}\theta & 0 &
	0\\
	-\varepsilon\frac{2\pi\kappa}{v\eta}\lambda_{2}\sin2\pi\frac{1}{v}\theta & 0 &
	0
	\end{array}
	\right)  .\nonumber
	\end{align}
	From (\ref{eq:der-encl-forw-n}) we have that for any $A\in\lbrack d\left(
	F_{\varepsilon}\right)  _{\phi}(N(\Lambda))]$%
	\begin{align}
	\inf\left\{  |A_{u}(0,x,y)|:|x|=1,|y|\leq1\right\}   &  \geq|\lambda
	_{1}|-\varepsilon\left(  \tau+\eta\right)  \frac{1}{2}\left(  1+\frac{\eta
	}{\tau}\right)  ,\label{eq:hen-temp1}\\
	\sup\left\{  |A_{s}(0,x,y):|x|\leq1,|y|\leq1|\right\}   &  \leq|\lambda
	_{2}|+\varepsilon\left(  \tau+\eta\right)  \frac{6}{1000}\left(  1+\frac{\tau
	}{\eta}\right)  .\nonumber
	\end{align}
	From (\ref{eq:zero-image}) and (\ref{eq:hen-temp1}), by Theorem
	\ref{th:diff-cover} (in our case $\varepsilon_{u}=\left\vert \lambda
	_{1}\right\vert /\tau,$ $\varepsilon_{s}=\left\vert \lambda_{2}\right\vert
	/\eta$) if we have
	\begin{align}
	|\lambda_{1}|-\varepsilon\left(  \tau+\eta\right)  \frac{1}{2}\left(
	1+\frac{\eta}{\tau}\right)   &  >1+\frac{1}{\tau}\left\vert \lambda
	_{1}\right\vert ,\label{eq:cover-est-henon1}\\
	|\lambda_{2}|+\varepsilon\left(  \tau+\eta\right)  \frac{6}{1000}\left(
	1+\frac{\tau}{\eta}\right)   &  <1-\frac{1}{\eta}\left\vert \lambda
	_{2}\right\vert , \label{eq:cover-est-henon2}%
	\end{align}
	then we have established (\ref{eq:forw-cover-hen}). The conditions
	(\ref{eq:cover-est-henon1}) and (\ref{eq:cover-est-henon2}) hold for all
	$\varepsilon\leq\frac{1}{2}$ with $\tau=3,$ $\eta=\frac{3}{40}.$

	To establish (\ref{eq:back-cover-hen}) we first compute%
	\begin{equation}
	F_{\varepsilon}^{-1}\left(  \theta,x,y\right)  =\left(  \theta-\omega,\frac
	{1}{b}y,x^{\prime}-1+\frac{a}{b^{2}}y^{2}-\varepsilon\cos\left(  2\pi\left(
	\theta-\omega\right)  \right)  \right)  . \label{eq:henon-inv-eps}%
	\end{equation}
	From the fact that $F_{\varepsilon}^{-1}$ is a rotation on $\mathbb{T}^{1}$
	and from the definition of the sets $V_{j},U_{i}$ (see (\ref{eq:atlas-henon}))
	we have that for any $V_{j}\subset U_{i_{0}}$ there exists $i_{1}$ such that
	$F_{\varepsilon}^{-1}(p^{-1}(V_{j})\cap|D|)\subset p^{-1}(U_{i_{1}}).$

	Once again, to simplify notations, we will consider $\left(  F_{\varepsilon
	}^{-1}\right)  _{\phi}$ instead of $\left(  F_{\varepsilon}^{-1}\right)
	_{i_{0}i_{1}}.$ Using the fact that $F_{0}^{-1}(\theta,x_{-},y_{-}%
	)=(\theta-\omega,x_{-},y_{-})$ we have%
	\begin{align*}
	\left(  F_{\varepsilon}^{-1}\right)  _{\phi}(\theta,0,0)  &  =\phi\circ
	F_{\varepsilon}^{-1}\circ\phi^{-1}\left(  \theta,0,0\right) \\
	&  =\phi\left(  F_{0}^{-1}(\theta,x_{-},y_{-})+\left(  0,0,\varepsilon
	\cos\left(  2\pi\left(  \frac{1}{v}\theta-\omega\right)  \right)  \right)
	\right) \\
	&  =\left(  \theta-v\omega,\left(  \Phi_{\varepsilon}(0,0)+\left(
	0,\varepsilon\cos\left(  2\pi\left(  \frac{1}{v}\theta-\omega\right)  \right)
	\right)  \right)  \right) \\
	&  =\left(  \theta-v\omega,-\frac{1}{\tau}\cos\left(  2\pi\left(
	\theta-\omega\right)  \right)  ,\frac{1}{\eta}\cos\left(  2\pi\left(  \frac
	{1}{v}\theta-\omega\right)  \right)  \right)  ,
	\end{align*}
	hence%
	\begin{equation}
	\left(  F_{\varepsilon}\right)  _{\phi}^{-1}(\Lambda,0,0)\subset\Lambda\times
	B\left(  0,\frac{1}{\tau}\right)  \times B\left(  0,\frac{1}{\eta}\right)  .
	\label{eq:zero-image-inv}%
	\end{equation}
	Let $g^{-}(\theta,x,y)=\left(  0,0,-\varepsilon\cos\left(  2\pi\left(
	\theta-\omega\right)  \right)  \right)  ,$ we then have $F_{\varepsilon}%
	^{-1}=F_{0}^{-1}+\varepsilon g^{-}$ and
	\begin{align}
	d\left(  F_{\varepsilon}^{-1}\right)  _{\phi}  &  =d\left(  F_{0}^{-1}\right)
	_{\phi}+\varepsilon\text{diag}(v\text{id},\Phi_{\varepsilon})dg^{-}(\phi
	^{-1}(\cdot))\text{diag}\left(  \frac{1}{v}\text{id},\Phi_{\varepsilon}%
	^{-1}\right) \label{eq:diff-F-phi-back}\\
	&  =\left(
	\begin{array}
	[c]{cc}%
	1 & 0\\
	0 & d\tilde{F}^{-1}%
	\end{array}
	\right)  +\left(
	\begin{array}
	[c]{ccc}%
	0 & 0 & 0\\
	-\frac{2\pi\varepsilon\kappa}{v\tau}\sin\left(  2\pi\left(  \frac{1}{v}%
	\theta-\omega\right)  \right)  & 0 & 0\\
	\frac{2\pi\varepsilon\kappa}{v\eta}\sin\left(  2\pi\left(  \frac{1}{v}%
	\theta-\omega\right)  \right)  & 0 & 0
	\end{array}
	\right)  .\nonumber
	\end{align}
	From (\ref{eq:der-encl-back}) we know that for any $A\in\lbrack D\left(
	F_{\varepsilon}^{-1}\right)  _{\phi}(N(\Lambda))]$ we have (let us note that
	the roles of the stable and unstable coordinates have been exchanged with
	respect to the forward map)%
	\begin{align*}
	\inf\left\{  |A_{u}(0,x,y)|:|x|=1,|y|\leq1\right\}   &  \geq\left\vert
	\frac{1}{\lambda_{2}}\right\vert -\varepsilon\left(  \tau\left\vert
	\lambda_{2}\right\vert +\eta|\lambda_{1}|\right)  \left(  \frac{\tau}{\eta
	}\frac{6}{10}+50\right)  ,\\
	\sup\left\{  |A_{s}(0,x,y)|:|x|\leq1,|y|\leq1\right\}   &  \leq\frac
	{1}{\lambda_{1}}+\varepsilon\left(  \tau\left\vert \lambda_{2}\right\vert
	+\eta|\lambda_{1}|\right)  \left(  \frac{6}{10}+50\frac{\eta}{\tau}\right)  .
	\end{align*}
	Hence from (\ref{eq:zero-image-inv}), by Theorem \ref{th:diff-cover} (in our
	case $\varepsilon_{u}=1/\eta$ and $\varepsilon_{s}=1/\tau$), if we have
	\begin{align}
	\left\vert \frac{1}{\lambda_{2}}\right\vert -\varepsilon\left(  \tau\left\vert
	\lambda_{2}\right\vert +\eta\lambda_{1}\right)  \left(  \frac{\tau}{\eta}%
	\frac{6}{10}+50\right)   &  >1+\frac{1}{\eta},\label{eq:cover-est-henon3}\\
	\frac{1}{\lambda_{1}}+\varepsilon\left(  \tau\left\vert \lambda_{2}\right\vert
	+\eta\lambda_{1}\right)  \left(  \frac{6}{10}+50\frac{\eta}{\tau}\right)   &
	<1-\frac{1}{\tau}, \label{eq:cover-est-henon4}%
	\end{align}
	then we have established (\ref{eq:back-cover-hen}). The conditions
	(\ref{eq:cover-est-henon3}) and (\ref{eq:cover-est-henon4}) hold for
	$\varepsilon\leq\frac{1}{2}$ with $\tau=3,$ $\eta=\frac{3}{40}.\,$

	\subsection{Verification of cone conditions}

	We will now use Theorem \ref{th:cone-ineq} to verify cone conditions. For any
	$x\in|D|$ we can choose a cone enclosing pair $(V_{j_{0}},U_{i_{0}})$ for $x$ (see
	(\ref{eq:atlas-henon})). From the fact that $F_{\varepsilon}$ is a rotation on
	$\mathbb{T}^{1}$ we have%
	\[
	\pi_{\theta}(F_{\varepsilon})_{\phi}(V_{j})\subset V_{j}+\omega v\text{ mod
	}v.
	\]
	If $x\in|D|,$ $F_{\varepsilon}(x)\in|D|,$ $\pi_{\theta}(\phi(F_{\varepsilon
	}(x)))\in\lbrack i,i+1]$ and $(V_{j_{0}},U_{i_{0}})$ is a cone enclosing pair for $x,$
	then we can find a $U_{i_{1}}$ for which we will both have
	\begin{align*}
	(V_{j}+\omega v\text{ mod }v)  &  \subset U_{i_{1}},\\
	\lbrack i-2,i+3]\text{ mod }v  &  \subset U_{i_{1}}.
	\end{align*}
	Setting $V_{j_{1}}=[i-2,i+3]$ mod $v$ we have found a cone enclosing pair $(V_{j_{1}%
	},U_{i_{1}})$ for $F_{\varepsilon}(x),$ for which $F_{\varepsilon}%
	(p^{-1}(V_{j})\cap|D|)\subset p^{-1}(U_{i_{1}})$. An analogous argument holds
	for $F_{\varepsilon}^{-1}.$

	Now we will verify (\ref{eq:cone-est1}) for the forward map $F_{\varepsilon}.$
	Our estimates will be independent from the choice of $j_{0},i_{0},i_{1},$
	hence as in Section \ref{sec:cover-ver} we will consider the map
	$(F_{\varepsilon})_{\phi}$ instead of $(F_{\varepsilon})_{i_{1}i_{0}}.$ From
	(\ref{eq:diff-F-phi}) and (\ref{eq:dF-forw}) we have
	\[
	d\left(  F_{\varepsilon}\right)  _{\phi}=\left(
	\begin{array}
	[c]{cc}%
	1 & 0\\%
	\begin{array}
	[c]{c}%
	\varepsilon\frac{2\pi\kappa}{v\tau}\lambda_{1}\sin2\pi\frac{1}{v}\theta\\
	-\varepsilon\frac{2\pi\kappa}{v\eta}\lambda_{2}\sin2\pi\frac{1}{v}\theta
	\end{array}
	& J+R_{\varepsilon}%
	\end{array}
	\right)  .
	\]
	This by (\ref{eq:der-encl-forw-n}) means that our constants from Theorem
	\ref{th:cone-ineq} are as follows%
	\[%
	\begin{array}
	[c]{lll}%
	C=1, & \varepsilon_{c}=0, & \mu=0,\\
	M=\left\vert \varepsilon\frac{2\pi\kappa}{v\tau}\lambda_{1}\right\vert , &
	\mathcal{A}=\lambda_{1}+\frac{1}{2}\varepsilon\left\vert \tau+\eta\right\vert
	, & \varepsilon_{u}=\frac{1}{2}\varepsilon\left\vert \frac{\eta}{\tau}\left(
	\tau+\eta\right)  \right\vert ,\\
	\alpha=\lambda_{1}-\frac{1}{2}\varepsilon\left\vert \tau+\eta\right\vert  &
	\varepsilon_{s}=\frac{6}{1000}\varepsilon\left\vert \frac{\tau}{\eta}\left(
	\tau+\eta\right)  \right\vert , & \beta=\left\vert \lambda_{2}\right\vert
	+\frac{6}{1000}\varepsilon\left\vert \tau+\eta\right\vert .
	\end{array}
	\]
	By choosing $v$ sufficiently large we can reduce $M$ arbitrarily close to
	zero. We also note that $\varepsilon_{c}=0$. This means that conditions
	(\ref{eq:cone-est1}) reduce to
	\begin{align}
	C^{2}  &  <m,\nonumber\\
	\alpha^{2}-\varepsilon_{s}^{2}-\varepsilon_{u}\mathcal{A}-\varepsilon_{s}\beta
	&  >m,\label{eq:cones-reduced}\\
	\beta^{2}+\varepsilon_{u}\mathcal{A}+\varepsilon_{s}\beta &  <m.\nonumber
	\end{align}
	These conditions hold for $\varepsilon\leq\frac{1}{2}$ with $\tau=3,$
	$\eta=\frac{3}{40}$ and $m=2$.

	Now we turn to the conditions for the inverse map. From
	(\ref{eq:diff-F-phi-back}) and (\ref{eq:dF-inv}) we have%
	\[
	d\left(  F_{\varepsilon}^{-1}\right)  _{\phi}=\left(
	\begin{array}
	[c]{cc}%
	1 & 0\\%
	\begin{array}
	[c]{c}%
	-\varepsilon\frac{2\pi\kappa}{v\tau}\sin\left(  2\pi\left(  \frac{1}{v}%
	\theta-\omega\right)  \right) \\
	\varepsilon\frac{2\pi\kappa}{v\eta}\sin\left(  2\pi\left(  \frac{1}{v}%
	\theta-\omega\right)  \right)
	\end{array}
	& J^{-1}+R_{\varepsilon}^{\prime}%
	\end{array}
	\right)  .
	\]
	This by (\ref{eq:der-encl-back}) means that our constants from Theorem
	\ref{th:cone-ineq} are as follows (note that for the inverse map the roles of
	the stable and unstable directions are reversed)%
	\[%
	\begin{array}
	[c]{ll}%
	C=1, & \varepsilon_{c}=0,\\
	M=\left\vert \varepsilon\frac{2\pi\kappa}{v\tau}\right\vert , & \mathcal{A}%
	=\frac{1}{\lambda_{2}}+50\varepsilon\left(  \tau\left\vert \lambda
	_{2}\right\vert +\eta|\lambda_{1}|\right)  ,\\
	\alpha=\frac{1}{\lambda_{2}}-50\varepsilon\left(  \tau\left\vert \lambda
	_{2}\right\vert +\eta|\lambda_{1}|\right)  , & \varepsilon_{s}=\varepsilon
	\left(  \tau\left\vert \lambda_{2}\right\vert +\eta|\lambda_{1}|\right)
	\left\vert 50\frac{\eta}{\tau}\right\vert ,\\
	\varepsilon_{u}=\varepsilon\left(  \tau\left\vert \lambda_{2}\right\vert
	+\eta|\lambda_{1}|\right)  \left\vert \frac{\tau}{\eta}\frac{6}{10}\right\vert
	,\quad & \beta=\frac{1}{\lambda_{1}}+\varepsilon\left(  \tau\left\vert
	\lambda_{2}\right\vert +\eta|\lambda_{1}|\right)  \frac{6}{10},\\
	\mu=0. &
	\end{array}
	\]
	By choosing $v$ sufficiently large we can once again reduce $M$ arbitrarily
	close to zero, which means that conditions (\ref{eq:cone-est1}) reduce to
	conditions same as (\ref{eq:cones-reduced}). These conditions hold for
	$\varepsilon\leq\frac{1}{2}$ with $\tau=3,$ $\eta=\frac{3}{40}$ and $m=200$.

	\subsection{The estimate of the region in which the manifold is contained}

	So far we have shown that for $\varepsilon\leq\frac{1}{2}$ our map
	$F_{\varepsilon}$ satisfies forward and backward cone conditions. This means
	that we have an invariant manifold inside of%
	\[
	|D|=\phi^{-1}\left(  N(\Lambda)\right)  .
	\]
	This gives us the following bounds
	\begin{align*}
	D  &  =\phi^{-1}\left(  N(\Lambda)\right) \\
	&  =\mathbb{T}^{1}\times\{\left(  x_{-},y_{-}\right)  +\Phi_{\varepsilon}%
	^{-1}\left(  \overline{B}(0,1)\times\overline{B}(0,1)\right)  \}\\
	&  \subset\mathbb{T}^{1}\times\{(x_{-},y_{-})+[-\varepsilon|\kappa|(\tau
	+\eta),\varepsilon|\kappa|(\tau+\eta)]\\
	&  \quad\times\lbrack-\varepsilon|\kappa|\left(  \tau|\lambda_{2}%
	|+\eta|\lambda_{1}|\right)  ,\varepsilon|\kappa|\left(  \tau|\lambda_{2}%
	|+\eta|\lambda_{1}|\right)  ]\}.
	\end{align*}
	With $\tau=3$ and $\eta=\frac{3}{40}$ we have $|D|\subset U_{\varepsilon}$
	(where $U_{\varepsilon}$ is given by (\ref{eq:U-epsilon})).

	%TCIDATA{OutputFilter=latex2.dll}
	%TCIDATA{Version=5.00.0.2606}
	%TCIDATA{LaTeXparent=0,0,online-edit.tex}
	\section{Properties of the local Brouwer degree\label{sec:deg-prop}}

	\textbf{Solution property. }\cite{L}%
	\[
	\text{If deg}(f,D,c)\neq0\text{ then there exists an }x\in D\text{
	with }f(x)=c.
	\]

	\textbf{Homotopy property.} \cite{L} Let $H:[0,1]\times
	D\rightarrow R^{n}$ be continuous. Suppose that
	\begin{equation}
	\bigcup_{\lambda\in\lbrack0,1]}H_{\lambda}^{-1}(c)\cap
	D\quad\text{is compact}
	\label{eq:union}%
	\end{equation}
	then
	\[
	\forall\lambda\in\lbrack0,1]\quad\text{deg}(H_{\lambda},D,c)=\text{deg}%
	(H_{0},D,c)
	\]
	If $[0,1]\times\overline{D}\subset$dom$(H)$ and $\overline{D}$ is
	compact, then (\ref{eq:union}) follows from the condition
	\[
	c\notin H([0,1],\partial D).
	\]

	\textbf{Degree property for affine maps.} \cite{L} Suppose that
	$f(x)=B(x-x_{0})+c$, where $B$ is a linear map and $x_{0}\in
	R^{n}.$ If the equation $B(x)=0$ has no nontrivial solutions (i.e
	if $Bx=0$, then $x=0$) and $x_{0}\in D$, then
	\begin{equation}
	\text{deg}(f,D,c)=\text{sgn}(\text{det}B). \label{eq:deg(f,C)
	=sgn(det(B))}%
	\end{equation}

	\textbf{Excision property.} \cite{L} Suppose that we have an open
	set $E$ such
	that $E\subset D$ and%
	\[
	f^{-1}(c)\cap D\subset E,
	\]
	then%
	\[
	\text{deg}(f,D,c)=\text{deg}(f,E,c).
	\]

%%%%%%%%%%%%%%%%%%%%%%%%%%%%%%%%%

\medskip
% The data information below will be filled by AIMS editorial staff
Received xxxx 20xx; revised xxxx 20xx.
\medskip

\end{document}